\documentclass[12pt,a4paper]{amsart}
\usepackage{color,graphicx,fullpage,amssymb,textcomp,array}
\usepackage[skip=10pt]{caption}
\usepackage[colorlinks=true]{hyperref}
\newtheorem{theorem}{Theorem}[section]
\newtheorem{lemma}[theorem]{Lemma}
\newtheorem{proposition}[theorem]{Proposition}
\newtheorem{corollary}[theorem]{Corollary}

\newtheorem{maintheorem}{Theorem}

\theoremstyle{definition}
\newtheorem{definition}[theorem]{Definition}
\newtheorem{remark}[theorem]{Remark}
\newtheorem{example}[theorem]{Example}

\newcommand{\Q}{\mathbb{Q}}
\newcommand{\Qp}{\mathbb{Q}_p}

\newcommand{\Zp}{\mathbb{Z}_p}
\newcommand{\N}{\mathbb{N}}
\newcommand{\R}{\mathbb{R}}
\newcommand{\Z}{\mathbb{Z}}
\newcommand{\C}{\mathbb{C}}
\newcommand{\dd}{\mathrm{d}}
\newcommand{\ii}{\mathrm{i}}
\newcommand{\M}{\mathcal{M}}
\newcommand{\X}{\mathfrak{X}}
\newcommand{\lie}{\mathrm{L}}

\newcommand{\ball}{\mathrm{B}}
\newcommand{\omegastd}{\omega_{\mathrm{std}}}
\newcommand{\letnpos}{Let $n$ be a positive integer}
\newcommand{\letpprime}{Let $p$ be a prime number}
\renewcommand{\le}{\leqslant}
\renewcommand{\ge}{\geqslant}
\AtBeginEnvironment{pmatrix}{\let\frac\dfrac}
\DeclareMathOperator{\ord}{ord}

\DeclareMathOperator{\vol}{vol}
\numberwithin{equation}{section}

\newenvironment{enumerate-roman}{\begin{enumerate}
		
	}{\end{enumerate}}

\title{Darboux's theorem in $p$-adic symplectic geometry}
\author[Luis Crespo, \'Alvaro Pelayo]{Luis Crespo\,\,\,\,\,\, \'Alvaro Pelayo}

\address{Luis Crespo,
	Departamento de Matem\'{a}ticas, Estad\'{i}stica y Computaci\'{o}n, Universidad de Cantabria, Av.~de Los Castros 48, 39005 Santander, Spain}
\email{luis.cresporuiz@unican.es}

\address{\'Alvaro Pelayo,
	Facultad de Ciencias Matem\'aticas,
	Universidad Complutense de Madrid, 28040 Madrid, Spain, and Real Academia de Ciencias Exactas, F\'isicas y Naturales, Madrid, Spain}
\email{alvpel01@ucm.es}

\begin{document}
	
\begin{abstract}
	We prove a non Archimedean Darboux's Theorem: any two symplectic forms on a $p$-adic analytic manifold are locally isomorphic. Understanding local problems such as the existence of flows or the normalization of singularities in the theory of integrable systems, is essential to understand the physics behind these systems. Our result tells us that the phase space defined by a $p$-adic manifold is locally standard, allowing us to concentrate on the equations defining the dynamics rather than on the space itself. Our proof uses a non Archimedean version of Moser's Path Method to push one symplectic form onto another one by a flow. A central technical contribution of the paper is the proof that the flow is given by a power series with \emph{non zero radius of convergence}, which requires geometric analytic estimates and does not follow from algebraic considerations. The main difficulty of the paper is analytic: an ODE will not have a long-time solution even in a neighborhood of a fixed point but we are able to find a condition for a long-time solution to exist, which allows us to tweak the ODE in Moser's Path Method so that it satifies the condition. As a global application, we derive a classification of second-countable $p$-adic analytic symplectic manifolds in terms of $p$-adic volume, which generalizes a classical theorem of J-P. Serre.
\end{abstract}

\maketitle

\section{Introduction}\label{sec:intro}

While Darboux's Theorem from the end of the 19th century is a cornerstone of classical symplectic geometry, its extension to non-Archimedean spaces has remained a challenge due to the difficulty on extending standard analytic tools, such as the classical Moser's Path Method, to non-Archimedean fields. In this paper, we develop a version of Moser's Path Method tailored for $p$-adic analytic manifolds. This requires addressing convergence issues inherent to $p$-adic power series to ensure the existence of local isomorphisms. Using this method as a stepping stone, we then prove that all $p$-adic symplectic forms are locally isomorphic, thereby establishing the lack of local invariants in the emerging field of $p$-adic symplectic geometry.

Throughout the paper we work with analytic manifolds ---instead of other types of analytic spaces which are fundamental in areas such as algebraic geometry--- because, from the angle of $p$-adic Lie groups and integrable systems, the classical notion of $p$-adic manifold as in Schneider's classical book \cite{Schneider} is sufficient and still very powerful. In fact, we believe that \cite{Schneider} provides an optimal language to treat problems in physics, where we expect $p$-adic symplectic geometry has a very good chance of playing a relevant role in the future, based on the deep connections it has with the $p$-adic analogues of string theory, quantum mechanics and cosmology. Because we are still working with smooth analytic manifolds (even if they are $p$-adic), this viewpoint also shows clearly the similarities between real and $p$-adic symplectic geometry, while at the same time capturing the physical intuition of crucial physical models like the $p$-adic Jaynes-Cummings model \cite{CrePel-JC} or the $p$-adic coupled angular momentum \cite{CrePel-angular}. These models naturally take place on $p$-adic analytic manifolds and their symmetries are described by $p$-adic Lie groups in the sense of Schneider.

\subsection{Summary of the paper and physical motivations}

In this paper we prove $p$-adic analogs of two fundamental results in geometry: Darboux's local classification of symplectic manifolds (1882) and Serre's classification of compact $p$-adic analytic manifolds (1965). In fact we will extend Serre's classification to the noncompact case as well, only assuming that $p$-adic analytic manifolds are second-countable. Our ``symplectic'' Serre's classification implies that all second-countable $p$-adic analytic symplectic manifolds admit global standard Darboux coordinates, which stands in very strong contrast with the real case.

We prove these results using a $p$-adic version of the Moser's Path Method, which we later also apply to study the physical models by Ablowitz-Ladik and Salerno of the Discrete Nonlinear Schrödinger equation.

In fact, symplectic geometry originated in the study of planetary motions and is deeply connected with classical mechanics. A main motivation for the present paper is to pursue a $p$-adic version of symplectic geometry, in the spirit of other important developments in mathematical physics involving the $p$-adic numbers, see Brekke--Freund \cite{BreFre}, Dragovich--Khrennikov--Kozyrev--Volovich \cite{DKKV}, Dragovich--Khrennikov--Kozyrev--Volovich--Zelenov \cite{DKKVZ}, Vladimirov--Volovich \cite{VlaVol}, and the references therein.
For treatments involving the $p$-adic numbers in string theory, see Brekke--Freund--Olson--Witten \cite{BFOW}, Freund--Olson \cite{FreOls}, Freund--Witten \cite{FreWit}, Fuquen-Tibatá--García-Compeán--Zúñiga-Galindo \cite{FGZ}, García-Compeán--López \cite{GarLop}, and Volovich \cite{Volovich}.
For applications to the theory of black holes, see Chen--Liu--Hung \cite{CLH}.
For a definition of the $p$-adic symplectic group, Heisenberg group and Maslov index, see Hu--Hu \cite{HuHu} and Zelenov \cite{Zelenov}.

\subsection{$p$-adic Moser's Path Method}\label{sec:moser-intro}

The following is a $p$-adic analytic analog of Moser's Path Method; see Appendix \ref{sec:real-moser} for a review of Moser's Path Method.

In the $p$-adic analytic setting, functions are given by power series on an $\ell$-dimensional $p$-adic analytic manifold:
\[\sum_{I=(i_1,\ldots,i_\ell)\in \N^\ell}a_I(x_1-x_{01})^{i_1}\ldots(x_\ell-x_{0\ell})^{i_\ell}\]
However, we cannot just apply Moser's Path Method to such functions, because the differential equation may not have a solution which converges in the entire manifold even when the coefficients of the equation do. In general an ODE does not have a long-time (power series) solution even in a neighborhood of a fixed point but we are able to find a condition for a long-time solution to exist, and tweak Moser's ODE in the Path Method so that it will satisfy this condition. This makes the statement of Moser's Path Method more technical than its real counterpart: concretely, we need to assume that the derivatives of the vector field vanish at the origin and that they converge in a bigger domain than the one we want for the diffeomorphism. Its proof is also significantly more involved so we divide it into several steps, spread out over sections \ref{sec:results}, \ref{sec:moser} and \ref{sec:darboux}.

This method provides the main technical tool required to prove the $p$-adic \emph{analytic} analogue of Darboux's Theorem (Theorem \ref{thm:darboux2}), which one cannot derive in the real case nor in the $p$-adic case resorting only to algebraic formalism, it is a nonlinear analytic problem (we explain this in Remark \ref{rem:algebraic} and Appendix \ref{sec:comparison}).

\pagebreak

\begin{maintheorem}[$p$-adic analytic Moser's Path Method]\label{thm:moser}
	Let $\ell$ and $k$ be integers with $0\le k\le \ell$ and $\ell\ge 1$. \letpprime. Let $\Zp$ denote the $p$-adic integers. Let $d=2$ if $p=2$ and otherwise $d=1$. Let $M$ be an $\ell$-dimensional $p$-adic analytic manifold and let $Q$ be a paracompact submanifold of $M$. Let $\{B_i:i\in I\}$ be a set of pairwise disjoint $p$-adic balls in $M$ which cover $Q$ (which always exists since $Q$ is paracompact) and for each $i\in I$ let $(x_1^i,\ldots,x_\ell^i)$ be coordinates on $B_i$. Let $a_S^i(x,t)$ be a $p$-adic power series in $t$, for every $i\in I$, $x\in B_i$ and $S\subset\{1,\ldots,\ell\}$ with $|S|=k$. Let $\{\omega_t\}_{t\in p^{-d}\Zp}$ be a $p$-adic analytic family of $k$-forms on $M$ given by
	\[(\omega_t)_x=\sum_{\substack{S\subset\{1,\ldots,\ell\} \\ |S|=k}} a_S^i(x,t)\bigwedge_{j\in S}\dd x_j^i\quad\forall t\in p^{-d}\Zp,\forall x\in B_i.\]
	Let $\lie$ denote the $p$-adic Lie derivative and suppose that, for every $m\in Q$, there exists an open neighborhood $V_m$ of $m$ contained in a ball $B_{i_m}$ such that the differential equation
	\[\frac{\dd}{\dd t}\omega_t+\lie_{X_{m,t}}\omega_t=0\]
	has a solution $X_{m,t}\in\mathfrak{X}(V_m)$ for $t\in p^{-d}\Zp$ which satisfies the following conditions:
	\begin{enumerate-roman}
		\item each component of $X_{m,t}$ is given by a $p$-adic power series in the coordinates $(x_1^{i_m},\ldots,x_\ell^{i_m})$ which converges for $x\in V_m$ and $t\in p^{-d}\Zp$;
		\item $X_{m,t}(m')=0$ for every $m'\in Q\cap V_m$;
		\item when the $p$-adic power series in {\normalfont(i)} are translated so that the origin is $m$, they do not contain terms with degree less than $2$, that is, the partial derivatives of the components of $X_{m,t}$ at $m$ are $0$.
	\end{enumerate-roman}
	Then there exists an open subset $U\subset M$ which contains $Q$ and a $p$-adic analytic family of embeddings $\{\psi_t:U\to U\}_{t\in\Zp}$ which fix $Q$, such that $\psi_t^*\omega_t=\omega_0|_U$.
\end{maintheorem}

\begin{remark}
	In Theorem \ref{thm:moser}, the $p$-adic analytic families of $k$-forms $\omega_t$ and of vector fields $X_{m,t}$, seen as power series in the variable $t$, must converge in $p^{-d}\Zp$ in order to apply Moser's Path Method, while the solution $\psi_t$ is only obtained for $t\in\Zp$.
\end{remark}

\subsection{Applications of $p$-adic Moser's Path Method to symplectic geometry}

\subsubsection{Applications to local symplectic geometry}

Theorem \ref{thm:moser} can be used, as happens in the real case, to prove Darboux's Theorem in the $p$-adic context:

\begin{maintheorem}[$p$-adic analytic Darboux's Theorem]\label{thm:darboux2}
	\letnpos. \letpprime. Let $(M,\omega)$ be a $2n$-dimensional $p$-adic analytic symplectic manifold and $m\in M$. There exist local $p$-adic analytic coordinates $(x_1,y_1,\ldots,x_n,y_n)$ around $m$ such that $\omega=\sum_{i=1}^n\dd x_i\wedge\dd y_i.$
\end{maintheorem}

\begin{figure}[h]
	\includegraphics{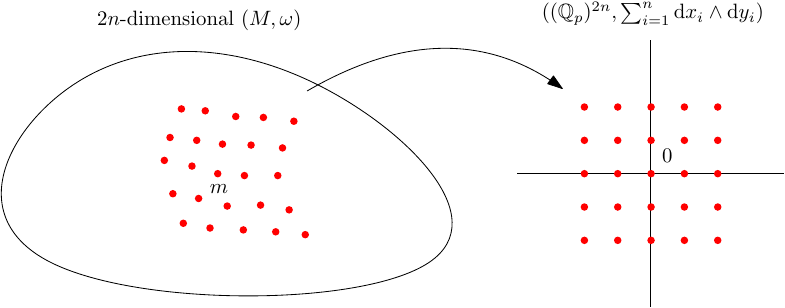}
	\caption{Illustration of the $p$-adic analytic Darboux's Theorem (Theorem \ref{thm:darboux2}).}
	\label{fig:darboux-2d}
\end{figure}

It follows from Theorem \ref{thm:darboux2} that the only local invariant of $p$-adic analytic symplectic manifolds is the dimension. See Figure \ref{fig:darboux-2d} for an illustration of Theorem \ref{thm:darboux2}.

\begin{remark}\label{rem:algebraic}
	In algebraic geometry it is typical to go from a statement at the tangent space at a point to a whole neighborhood at the level of \emph{formal power series}, but this is insufficient for $p$-adic \emph{analytic} symplectic geometry because, while a solution may exist as a formal sum, the condition $\phi^*\omega=\sum_{i=1}^n\dd x_i\wedge\dd y_i$ in Theorem \ref{thm:darboux2} is a system of nonlinear PDEs where a linear solution at a point does not automatically imply that there is a convergent analytic solution in a neighborhood of \emph{nonzero} radius: to prove it is the crucial point of our paper. Theorem \ref{thm:darboux2} shows it is the case, but it cannot be derived only from algebraic techniques (at least in any way we can see): analytic estimates, as we did, are genuinely required, as in the technical proofs of Section \ref{sec:results}. We refer to Appendix \ref{sec:comparison} for an explanation of how this is related to a classical statement of Milnor-Husemoller, from which one may be tempted to deduce Theorem \ref{thm:darboux2} (but which is not possible directly).
\end{remark}

Using the $p$-adic Moser's Path Method we can generalize Theorem \ref{thm:darboux2} to neighborhoods of compact submanifolds, and obtain a $p$-adic analog of the Darboux-Weinstein's Theorem \cite[Theorem 4.1]{Weinstein-submanifolds}.

\begin{maintheorem}[$p$-adic analytic Darboux-Weinstein's Theorem]\label{thm:darboux-weinstein}
	\letnpos. \letpprime. Let $M$ be a $2n$-dimensional $p$-adic analytic manifold and let $Q$ be a compact submanifold of $M$. Let $\omega_0$ and $\omega_1$ be two $p$-adic analytic symplectic forms on $M$ which coincide on $Q$. Then there exist open neighborhoods $U_0,U_1$ of $Q$ and a $p$-adic analytic diffeomorphism $\psi:U_0\to U_1$ which fixes $Q$ and such that $\psi^*\omega_1=\omega_0$.
\end{maintheorem}

\subsubsection{Applications to global symplectic geometry}

In what follows, when we say that two $p$-adic manifolds are symplectomorphic, it is assumed that they are so by means of a $p$-adic analytic symplectomorphism.

\begin{maintheorem}\label{thm:volumes}
	Let $v$ be a positive real number. Let $d$ be a positive integer. Then the following statements hold:
	\begin{enumerate}
		\item Let $\M_{d,v}\subset(\Qp)^d$ be defined as follows: if
		$v=\sum_{k=-\infty}^m a_kp^{dk}$
		is the expression of $v$ in base $p^d$ with infinitely many nonzero digits (for example $1$ in base $9$ becomes $0.888\ldots$), then,
		\[\M_{d,v}=\bigcup_{k=-\infty}^{m}\bigcup_{i=1}^{a_k}\ball_{k,i}\subset(\Qp)^d\]
		where $\ball_{ki}$ is the $d$-dimensional $p$-adic ball defined by \[\ball_{k,i}=\ball((i_1p^{-k-1},i_2p^{-k-1},\ldots,i_dp^{-k-1}),p^k),\] which are all pairwise disjoint, with $i=\sum_{j=0}^{d-1} i_{j+1}p^j$ and $0\le i_j\le p-1$. Then $\M_{d,v}$ is a noncompact $d$-dimensional $p$-adic analytic manifold with $p$-adic volume $v$. See Figure \ref{fig:balls-ref} for the balls $\ball_{k,i}$.
		\item If, in addition, $v$ is rational and its denominator is a power of $p$, then let $\M_{d,v}'\subset(\Qp)^d$ be defined as follows: if
		$v=\sum_{k=\ell}^m a_kp^{dk}$
		is the expression of $v$ in base $p^d$ with finitely many nonzero digits, which exists because $v$ is rational and its denominator is a power of $p$, then
		\[\M_{d,v}'=\bigcup_{k=\ell}^{m}\bigcup_{i=1}^{a_k}\ball_{k,i}\subset(\Qp)^d.\]
		Then $\M_{d,v}'$ is a compact $d$-dimensional $p$-adic analytic manifold with $p$-adic volume $v$.
	\end{enumerate}
	If $d=2n$ for $n$ integer, endow $\M_{d,v}$ and $\M_{d,v}'$ with the standard symplectic form $\sum_{i=1}^n\dd x_i\wedge\dd\xi_i$. If in the definition of $\M_{d,v}$ and $\M_{d,v}'$ we change the balls by other pairwise disjoint balls of the same radius the resulting $p$-adic analytic manifolds are symplectomorphic.
\end{maintheorem}

Now we turn to the symplectic version of Serre's result Theorem \ref{thm:serre}. The \emph{$p$-adic volume} which we use here is the $p$-adic measure with respect to the form $\omega^n$, where $\omega$ is the symplectic form, analogously to the definition in the real case.

\begin{maintheorem}[Symplectic Serre's classification of $p$-adic analytic manifolds]\label{thm:global-coord}
	\letnpos. \letpprime. Let $(M,\omega)$ be a second-countable $2n$-dimensional $p$-adic analytic symplectic manifold. Let $v$ be the volume of $(M,\omega)$ and let $\M_{2n,v}$ and $\M_{2n,v}'$ be as defined in Theorem \ref{thm:volumes}. Endow $(\Qp)^{2n}$, $\M_{2n,v}$ and $\M_{2n,v}'$ with the standard symplectic form $\sum_{i=1}^n\dd x_i\wedge\dd\xi_i$.
	\begin{enumerate}
		\item[(1a)] If $M$ is noncompact and $v=\infty$, $(M,\omega)$ is symplectomorphic to $(\Qp)^{2n}$.
		\item[(1b)] If $M$ is noncompact and $v$ is finite, $(M,\omega)$ is symplectomorphic to $\M_{2n,v}$.
		\item[(1c)] If $M$ is compact, then $v$ is a rational number with a power of $p$ as denominator and $(M,\omega)$ is symplectomorphic to $\M_{2n,v}'$.
		\item[(2)] $\M_{2n,v}$ and $\M_{2n,v'}$ are symplectomorphic if and only if $v=v'$. The same happens with $\M_{2n,v}'$ and $\M_{2n,v'}'$.
	\end{enumerate}
\end{maintheorem}

Recall that $\M_{2n,v}$ and $\M_{2n,v}'$ are defined as unions of balls, hence Theorem \ref{thm:global-coord} is a symplectic version of Serre's Classification (Theorem \ref{thm:serre}); the difference is that the symplectic case needs an unbounded number of balls. Theorem \ref{thm:global-coord} implies that there exists a $p$-adic analytic symplectic embedding
$\varphi:(M,\omega)\to((\Qp)^{2n},\sum_{i=1}^n\dd x_i\wedge\dd\xi_i).$
Hence, there always exist global Darboux coordinates.

\begin{maintheorem}\label{thm:classification}
	\letnpos. \letpprime. Let $(M,\omega)$ and $(M',\omega')$ be second-countable $2n$-dimensional $p$-adic analytic symplectic manifolds with the same $p$-adic volume, finite or infinite, and which are both compact or both noncompact. Then $(M,\omega)$ is symplectomorphic to $(M', \omega')$.
\end{maintheorem}

Theorem \ref{thm:classification} stands in very strong contrast with real symplectic geometry where there is an extensive theory of other invariants, such as symplectic capacities.

\subsubsection{Application of the $p$-adic Moser's Path Method to the Ablowitz-Ladik and Salerno models}

In Section \ref{sec:examples} we apply the method of proof of Theorem \ref{thm:darboux2} to a $p$-adic analog of the Ablowitz-Ladik's Model and Salerno's Model, which are generalizations of the Discrete Nonlinear Schrödinger Model. In the real case, after changing coordinates so that the symplectic form takes the standard form, the Hamiltonian of these models takes the same form as the mentioned Discrete Nonlinear Schrödinger model (see \cite{CPP} for details). It is not possible to derive a closed form for the coordinate change, but some intermediate results can be obtained. We will obtain those results for the $p$-adic case (Proposition \ref{prop:dNLS}), with the same expressions.

Ablowitz-Ladik's and Salerno's models use the same nonstandard symplectic form:
\[\omega=\ii\sum_{j\in J}\frac{\dd \psi_j\wedge\dd\bar{\psi}_j}{1+\mu|\psi_j|^2},\]
where the $\psi_j$ are complex coordinates and $\mu\in\R$. In the $p$-adic case, we are able to recover the result from \cite{CPP}, which tells us that this symplectic form can be taken into the standard one
$\omegastd=\ii\sum_{j\in J}^n\dd \psi_j\wedge\dd\bar{\psi}_j,$
using as a transformation the flow of the vector field
\[X_t=\frac{1+\nu(x_1^2+y_1^2)}{1+(1-t)\nu(x_1^2+y_1^2)}\left(1-\frac{\log(1+\nu(x_1^2+y_1^2))}{2\nu(x_1^2+y_1^2)}\right)\begin{pmatrix}
	x_1 \\ y_1
\end{pmatrix},\]
for $\nu=\mu/2$.

\subsection*{Structure of the paper}

Section \ref{sec:results} proves some intermediate results which we need in order to derive Moser's Path Method and prove Darboux's Theorem. Section \ref{sec:moser} derives the $p$-adic Moser's Path Method (Theorem \ref{thm:moser}). Section \ref{sec:darboux} proves the $p$-adic Darboux's Theorem (Theorem \ref{thm:darboux2}). Section \ref{sec:darbwein} proves the $p$-adic Darboux-Weinstein's Theorem (Theorem \ref{thm:darboux-weinstein}). Section \ref{sec:classification} proves Theorems \ref{thm:volumes} and \ref{thm:global-coord}. Finally, Section \ref{sec:examples} applies Darboux's Theorem to the $p$-adic version of the Ablowitz-Ladik and Salerno models. The appendices recall the concepts we need about $p$-adic numbers, $p$-adic analytic functions, $p$-adic analytic manifolds and $p$-adic symplectic geometry (Appendix \ref{sec:prelim}), the real Moser's Path Method (Appendix \ref{sec:real-moser}), the classical Serre's classification (Appendix \ref{sec:serre}) and a comparison with results from the book by Milnor and Husemoller (Appendix \ref{sec:comparison}).

\subsection*{Funding}

The first author is funded by grant PID2022-137283NB-C21 of MCIN/AEI/ 10.13039/501100011033 / FEDER, UE. The second author is funded by a FBBVA (Bank Bilbao Vizcaya Argentaria Foundation) Grant for Scientific Research Projects with title \textit{From Integrability to Randomness in Symplectic and Quantum Geometry}.

\subsection*{Acknowledgments}

The second author thanks the Dean of the School of Mathematical Sciences Antonio Br\'u and the Chair of the Department of Algebra, Geometry and Topology at the Complutense University of Madrid, Rutwig Campoamor, for their support and excellent resources he is being provided with to carry out the FBBVA project. The second author also thanks Luis Crespo and Francisco Santos for the hospitality during July and August 2025 at the University of Cantabria when part of this paper was written.

\noindent \textbf{Declarations}

\noindent \textbf{Conflict of interest.} The authors are not aware of a conflict of interest on their side related to this article.

\noindent \textbf{Data availability.} No data was used or generated in the production of this article.

\section{Technical results on $p$-adic initial value problems, power series and forms on manifolds}\label{sec:results}

In this section we prove four technical lemmas which we will need to prove the main theorems (Theorems \ref{thm:moser}, \ref{thm:darboux2}, \ref{thm:darboux-weinstein}, \ref{thm:volumes}, \ref{thm:global-coord}, \ref{thm:classification}). We refer to Appendix \ref{sec:prelim} for a quick review of the concepts and results from $p$-adic analysis and $p$-adic geometry we need in this section (in particular see Figures \ref{fig:ball} and \ref{fig:ball2} for representations of $p$-adic balls like the ones which appear in Lemma \ref{lemma:initial} below).

\begin{lemma}[Multivariable $p$-adic initial value problem]\label{lemma:initial}
	Let $\ell$ be a positive integer. \letpprime. Let $d=2$ if $p=2$ and otherwise $d=1$. Let $x_0\in\Qp$, $R$ a $p$-adic absolute value, $U_1=\ball(x_0,R)$ and $U_2=\ball(x_0,p^dR)$. Let $v\in(\Qp)^\ell$ and let $V$ be a $p$-adic ball in $(\Qp)^\ell$ centered at $v$. For every $i\in\{1,\ldots,\ell\}$ let $f_i\in\Qp[[x,y_1,\ldots,y_\ell]]$ be a $p$-adic power series converging in $U_2\times V$ such that:
	\begin{itemize}
		\item for every $i\in\{1,\ldots,\ell\}$ and every $x\in U_2$ we have that
		\[f_i(x,v)=0\]
		and,
		\item for all $i,j\in\{1,\ldots,\ell\}$ and every $x\in U_2$, \[\frac{\partial f_i}{\partial y_j}(x,v)=0.\]
	\end{itemize}
	Then there exist a $p$-adic ball $V'\subset V$ centered at $v$ and $p$-adic power series \[(g_1,\ldots,g_\ell)\in\Qp[[x,z_1,\ldots,z_\ell]]^\ell:U_1\times V'\to V'\] such that:
	\begin{itemize}
		\item $|g_i(x,z)|_p\le\|z-v\|_p$ for every $x\in U_1$, $z\in V'$ and $i\in\{1,\ldots,\ell\}$, and
		\item for any $(y_{0,1},\ldots,y_{0,\ell})\in V'$, the $p$-adic initial value problem of the form
		\begin{equation}\label{eq:initial}
			\left\{\begin{aligned}
				&\frac{\dd y_1}{\dd x}=f_1(x,y_1,\ldots,y_\ell); \\
				&\ldots \\
				&\frac{\dd y_\ell}{\dd x}=f_\ell(x,y_1,\ldots,y_\ell); \\
				&y_1(x_0)=y_{0,1},\,\,\ldots,\,\,y_\ell(x_0)=y_{0,\ell},
			\end{aligned}\right.
		\end{equation}
		has a unique solution given by $y_i(x)=g_i(x,y_{0,1},\ldots,y_{0,\ell})$.
	\end{itemize}
\end{lemma}

\begin{proof}
	\textit{Step 1: reduction to a canonical form.}
	By translating the power series if necessary, we may assume that $x_0=0$ and $v=0$. This does not change the convergence radii of the series, because the convergence radius of a $p$-adic power series is invariant by translation. After translating, $f_i(x,0,\ldots,0)$ is $0$ for every $i\in\{1,\ldots,\ell\}$ and every $x\in U_2$, and the same happens with its partial derivatives respect to $y_j$, hence the series $f_i$ have no terms with degree $0$ or $1$ in $y_1,\ldots,y_\ell$.
	
	Now we multiply the variables by some constants so that the balls $U_2$ and $V$ become $\Zp$ and $(\Zp)^\ell$, in particular, the radius of $U_1$ is $R=p^{-d}$.
	
	\textit{Step 2: choice of $V'$.} Let $c_i(s_0,\ldots,s_\ell)$ be the coefficient of $f_i$ with degree $s_0,s_1,\ldots,s_\ell$ in $x,y_1,\ldots,y_\ell$. Since $f_i$ converges in $(\Zp)^{\ell+1}$, the coefficients of total degree $s$ tend to $0$ when $s$ tends to infinity, that is,
	\begin{equation}\label{eq:limit}
		\lim_{s\to\infty}\max\Big\{|c_i(s_0,\ldots,s_\ell)|_p:i\in\{1,\ldots,\ell\},\,\,s_0,s_1,\ldots,s_\ell\in\N,\,\,\sum_{h=0}^\ell s_h=s\Big\}=0.
	\end{equation}
	In particular, the set
	\[\Big\{|c_i(s_0,\ldots,s_\ell)|_p:i\in\{1,\ldots,\ell\},\,\,s_0,s_1,\ldots,s_\ell\in\N\Big\}\]
	is bounded. This allows us to define, for each $s\in\N$,
	\[C_s=\max\Big\{|c_i(s_0,\ldots,s_\ell)|_p:i\in\{1,\ldots,\ell\},\,\,s_0,s_1,\ldots,s_\ell\in\N,\,\,\sum_{h=1}^\ell s_h=s\Big\}.\]
	Since $f_i$ has no constant or linear terms in the variables $y_j$, we have that
	\[C_0=C_1=0.\]
	Also, \eqref{eq:limit} implies that
	\begin{equation}\label{eq:limit2}
		\lim_{s\to\infty} C_s=0.
	\end{equation}
	Expression \eqref{eq:limit2} implies that only a finite number of the $C_s,s\in\N,$ can be greater than $1$. Hence we may take $r$ small enough so that
	\[r^{s-1}C_s\le p^{2-s}\]
	for all $s\ge 2$. We define $V'=\ball(0,r).$
	
	\textit{Step 3: uniqueness of solution.}
	We now check that the solution to \eqref{eq:initial} is indeed unique for an initial value in $V'$ and has the desired form.
	
	We want that the solution to \eqref{eq:initial} for a given $(y_{0,1},\ldots,y_{0,\ell})\in V'$ satisfies that $y_i(x)=g_i(x,y_{0,1},\ldots,y_{0,\ell})$ for some power series $g_i$. Let $a_{i,j}\in\Qp[[z_1,\ldots,z_\ell]]$ be given by
	\[g_i(x,z)=\sum_{j=0}^\infty a_{i,j}(z)x^j.\]
	In particular,
	\[a_{i,0}(y_{0,1},\ldots,y_{0,\ell})=g_i(0,y_{0,1},\ldots,y_{0,\ell})=y_i(0)=y_{0,i},\]
	which implies $a_{i,0}(z)=z_i$.
	We also have that
	\[\sum_{j=0}^\infty a_{i,j}(z)jx^{j-1}=\frac{\partial g_i(x,z)}{\partial x}=f_i(x,g_i(x,z))=f_i\left(x,\sum_{j=0}^\infty a_{1,j}(z)x^j,\ldots,\sum_{j=0}^\infty a_{n,j}(z)x^j\right).\]
	
	We expand the right-hand side in terms of the coefficients $c_i(s_0,\ldots,s_\ell)$:
	\begin{equation}\label{eq:pre-relation}
		\sum_{j=0}^\infty a_{i,j}(z)jx^{j-1}=\sum_{s_0,s_1,\ldots,s_\ell\in\N}c_i(s_0,\ldots,s_\ell)x^{s_0}\prod_{h=1}^\ell\left(\sum_{j=0}^\infty a_{h,j}(z)x^j\right)^{s_h}.
	\end{equation}
	
	Let $S$ be the set of tuples of the form
	\[t=(s_0,(h_1,\ldots,h_s),(j_1,\ldots,j_s))\]
	where $s,s_0\in\N$ and, for each $b\in\{1,\ldots,s\}$, $h_b\in\{1,\ldots,\ell\}$, $h_b\le h_{b+1}$, and $j_b\in\N$. Given $t\in S$, we call $s(t)=s$, $s_0(t)=s_0$, $h_b(t)=h_b$, $j_b(t)=j_b$, \[s_h(t)=\#\{b\in\{1,\ldots,s\}:h_b=h\}\text{ for } h\in\{1,\ldots,\ell\},\] and \[\deg(t)=s_0+\sum_{b=1}^{s(t)} j_b.\] Now \eqref{eq:pre-relation} implies that
	\begin{equation}\label{eq:relation}
		\sum_{j=0}^\infty a_{i,j}(z)jx^{j-1}=\sum_{t\in S} c_i(s_0(t),s_1(t),\ldots,s_\ell(t))x^{\deg(t)}\prod_{b=1}^{s(t)} a_{h_b(t),j_b(t)}(z).
	\end{equation}
	
	Indeed, the product in \eqref{eq:relation} for a given $t\in S$ has exactly $s_h(t)$ factors of the form $a_{h,j}$ for a fixed $h$ and any $j$, which must come from the $h$-th factor of the product in \eqref{eq:pre-relation} with $s_h=s_h(t)$ and as such contribute with $x^j$ to the product. The powers of $x$, together with $x^{s_0(t)}$, multiply to $x^{\deg(t)}$.

	Now, for $k\in\N$, we take the part of degree $k$ in $x$:
	\begin{equation}\label{eq:relation2}
		(k+1)a_{i,k+1}(z)=\sum_{\substack{t\in S \\ \deg(t)=k}} c_i(s_0(t),s_1(t),\ldots,s_\ell(t))\prod_{b=1}^{s(t)} a_{h_b(t),j_b(t)}(z).
	\end{equation}
	If $\deg(t)=k$, we must have $j_b(t)\le k$, hence this gives an expression for $a_{i,k+1}$ in terms of $a_{i,j}$ for smaller $j$.
	
	\textit{Step 4: convergence of the solution.}
	It is left to prove that $g_i$ converges in $U_1\times V'$ and the image of $(g_1,\ldots,g_\ell)$ is contained in $V'$. In order to do this, we first prove that $a_{i,j}$ converges in $V'$ and
	\begin{equation}\label{eq:bound}
		|a_{i,j}(z)|_p\le\frac{\|z\|_p}{|j!|_p}
	\end{equation}
	for every $z\in V'$.
	
	We show this by induction on $j$. For $j=0$, $a_{i,0}(z)=z_i$ converges and
	\[|a_{i,0}(z)|_p=|z_i|_p\le \|z\|_p.\]
	Now we suppose it holds for $0\le j\le k$ and prove it for $j=k+1$. In order to do this, we define, for each $i\in\{1,\ldots,\ell\}$ and $s\in\N$,
	\[A_{i,k,s}(z)=\sum_{\substack{t\in S \\ \deg(t)=k \\ s(t)=s}} c_i(s_0(t),s_1(t),\ldots,s_\ell(t))\prod_{b=1}^s a_{h_b(t),j_b(t)}(z).\]
	By induction hypothesis, $a_{h_b(t),j_b(t)}$ is defined in all $V'$. The number of $t\in S$ with a fixed degree and a fixed $s(t)$ is finite, hence $A_{i,k,s}$ is also defined in all $V'$. It also satisfies that
	\[|A_{i,k,s}(z)|_p\le\max_{\substack{t\in S \\ \deg(t)=k \\ s(t)=s}} |c_i(s_0(t),s_1(t),\ldots,s_\ell(t))\prod_{b=1}^s a_{h_b(t),j_b(t)}(z)|_p.\]
	Suppose that the maximum is attained by $t\in S$. By the definition of $C_s$,
	\[|A_{i,k,s}(z)|_p\le C_s\prod_{b=1}^s|a_{h_b(t),j_b(t)}(z)|_p.\]
	If $s\le 1$, $C_s=0$, which implies $A_{i,k,s}(z)=0$. Otherwise, by induction hypothesis,
	\[|A_{i,k,s}(z)|_p\le C_s\prod_{b=1}^s\frac{\|z\|_p}{|j_b(t)!|_p}=\frac{C_s\|z\|_p^s}{|j_1(t)!\ldots j_s(t)!|_p}.\]
	Since $z\in V'$ and $s\ge 2$,
	\[|A_{i,k,s}(z)|_p\le \frac{C_sr^{s-1}\|z\|_p}{|j_1(t)!\ldots j_s(t)!|_p}.\]
	Using that $C_s r^{s-1}\le p^{2-s}$,
	\[|A_{i,k,s}(z)|_p\le \frac{p^{2-s}\|z\|_p}{|j_1(t)!\ldots j_s(t)!|_p}.\]
	It is known that $j_1(t)!\ldots j_s(t)!$ divides $(j_1(t)+\ldots+j_s(t))!$, which in turn divides $\deg(t)!=k!$, and we have that
	\begin{equation}\label{eq:bound2}
		|A_{i,k,s}(z)|_p\le\frac{p^{2-s}\|z\|_p}{|k!|_p}.
	\end{equation}
	By equation \eqref{eq:relation2},
	\[(k+1)a_{i,k+1}(z)=\sum_{s=0}^\infty A_{i,k,s}(z).\]
	By equation \eqref{eq:bound2}, the absolute values of the terms in this sum tend to $0$, hence the sum converges, and
	\[|(k+1)a_{i,k+1}(z)|_p\le\max_{s\in\N}|A_{i,k,s}(z)|_p\le\frac{\|z\|_p}{|k!|_p},\]
	which implies
	\[|a_{i,k+1}(z)|_p\le\frac{\|z\|_p}{|(k+1)!|_p}.\]
	The induction step is complete.
	
	Now let $x\in U_1$ and $z\in V'$. We have that
	\[|a_{i,j}(z)x^j|_p\le|a_{i,j}(z)|_p|x|_p^j\le\frac{\|z\|_p}{|j!|_p}p^{-dj},\]
	where the first inequality follows from the strong triangle inequality and the second from \eqref{eq:bound} and $x\in U_1$. We also have that
	\[\ord_p(j!)=\sum_{k=1}^\infty\left\lfloor\frac{j}{p^k}\right\rfloor\le\sum_{k=1}^\infty\frac{j}{p^k}=\frac{j}{p-1},\]
	which implies $|a_{i,j}(z)x^j|_p\le \|z\|_pp^{j/(p-1)-dj}$. Since $1/(p-1)-d$ is negative, the series
	\[g_i(x,z)=\sum_{j=0}^\infty a_{i,j}(z)x^j\]
	converges and
	\[|g_i(x,z)|_p\le\|z\|_p.\]
	In particular, the image of $(g_1,\ldots,g_\ell)$ is contained in $V'$, as we wanted.
\end{proof}

\begin{remark}
	Lemma \ref{lemma:initial} is similar in spirit to our previous result \cite[Proposition A.9]{CrePel-JC}, but that result was weaker and in particular it only concerned one variable.
\end{remark}

\begin{remark}
	The proof of Lemma \ref{lemma:initial} does not work if we do not assume that the derivatives of the functions vanish at $v$. This happens because in that case we do not have $C_1=0$, and in particular we might have $C_1>p$ and we cannot find $r$ small enough such that $r^{s-1}C_s\le p^{2-s}$ for all $s$. We may use instead $r^{s-1}C_s\le p^{1-s}C_1$, which would lead to a similar proof, but then the solution only converges when $|x|_p\le 1/C_1$ instead of $|x|_p\le 1/p$, which is not useful for our purposes.
\end{remark}

\begin{example}
	Consider the differential equation
	\[y'=\frac{y^2}{1-x}\]
	with $x_0=0$ and $y(0)=y_0$. Its power series converges for $x\in p\Zp$, hence we can take $U_2=p\Zp$. Its solution can be computed explicitly:
	\[y=\frac{1}{\frac{1}{y_0}+\log(1-x)}.\]
	If we solve the equation using the method in the proof of Lemma \ref{lemma:initial}, we set $y$ to a power series
	\[y=\sum_{j=0}^\infty a_jx^j\]
	and substitute in the equation, obtaining
	\begin{align*}
		\sum_{j=0}^\infty a_jjx^{j-1} & =\frac{1}{1-x}\left(\sum_{j=0}^\infty a_jx^j\right)^2 \\
		& =\left(\sum_{h=0}^\infty x^h\right)\left(\sum_{i=0}^\infty\sum_{j=0}^\infty a_ia_jx^{i+j}\right) \\
		& =\sum_{k=0}^\infty x^k \sum_{i=0}^k \sum_{j=0}^{k-i} a_ia_j
	\end{align*}
	which gives us a recurrence relation for the $a_j$:
	\[a_{k+1}=\frac{1}{k+1}\sum_{i=0}^k \sum_{j=0}^{k-i} a_ia_j,\]
	with initial term $a_0=y_0$. Its first terms are
	\[\left\{
	\begin{aligned}
		a_0 & =y_0; \\
		a_1 & =y_0^2; \\
		a_2 & =\frac{y_0^2}{2}+y_0^3; \\
		a_3 & =\frac{y_0^2}{3}+y_0^3+y_0^4.
	\end{aligned}
	\right.\]
	Lemma \ref{lemma:initial} tells us that this power series converges at least in $U_1=p^{1+d}\Zp$.
\end{example}

\begin{lemma}\label{lemma:rational}
	\letpprime. Let $r$ be a $p$-adic absolute value. Let $f\in\Qp(t)$ be a rational function with $p$-adic coefficients. Let $d$ be the degree of the denominator of $f$. For $i\in\{0,\ldots,d\}$, let $b_i$ be the coefficient of degree $i$ in the denominator of $f$. Suppose that, for all $i\in\{1,\ldots,d\}$, we have that
	\[\left|\frac{b_i}{b_0}\right|_p\le (pr)^{-i}.\]
	Then the $p$-adic power series of $f$ converges in $\ball(0,r)$.
\end{lemma}

\begin{proof}
	It is enough to prove it for the case when the numerator is $1$, because multiplying a power series converging in $\ball(0,r)$ times a polynomial results in another power series converging in the same ball.
	
	For every $i\in\N$, let $a_i$ be the coefficient of degree $i$ in the power series of $f$. We can derive a recurrence for the $a_i$ in terms of the $b_i$:
	\[\left\{\begin{aligned}
		a_0b_0 & =1; \\
		\sum_{i=0}^d a_{k-i}b_i & =0\text{ for }k\in\N.
	\end{aligned}\right.\]
	This implies that
	\[\left\{\begin{aligned}
		a_0 & =\frac{1}{b_0}; \\
		a_k & =-\sum_{i=1}^d a_{k-i}\frac{b_i}{b_0}\text{ for }k\in\N.
	\end{aligned}\right.\]
	Hence
	\begin{align*}
		|a_k|_p & \le\max\left\{\left|a_{k-i}\frac{b_i}{b_0}\right|_p:i\in\{1,\ldots,d\}\right\} \\
		& \le\max\left\{\frac{|a_{k-i}|_p}{(pr)^i}:i\in\{1,\ldots,d\}\right\}.
	\end{align*}
	This implies, by induction on $k$, that
	\[|a_k|_p\le\frac{|a_0|_p}{(pr)^k},\]
	for any $k\in\N$, so
	\[r^k|a_k|_p\le\frac{|a_0|_p}{p^k},\]
	which tends to $0$ when $k$ tends to infinity. Hence the power series of $f$ converges in $\ball(0,r)$, as we wanted.
\end{proof}

\begin{example}
	Consider the function
	\[f(x)=\frac{1}{x^2+1}.\]
	If $-1$ is not a square in $\Qp$, $f(x)$ is defined in all $\Qp$. However, its power series
	\[f(x)=1-x^2+x^4-x^6+\ldots\]
	converges only in the $p$-adic balls with radius smaller than $1$. Lemma \ref{lemma:rational} explains the situation: the coefficients of the denominator do not decrease fast enough. The same example also works with $\R$ instead of $\Qp$.
\end{example}

\begin{lemma}\label{lemma:rational2}
	\letpprime. Let $r$ be a $p$-adic absolute value. Let $f\in\Qp[[x_1,\ldots,x_\ell]](t)$ be a rational function in $t$ with $p$-adic power series in $x_1,\ldots,x_\ell$ as coefficients. Let $U\subset(\Qp)^\ell$ open and $A\subset U$. Suppose that, if $(x_1,\ldots,x_\ell)\in A$, the denominator of $f(t,x_1,\ldots,x_\ell)$ is independent of $t$ and does not vanish in $A$. Then there exists $V\subset U$ open such that $A\subset V$ and, for every $y\in V$, the power series of $f(t,y)$ as a function of $t$ converges in $\ball(0,r)$.
\end{lemma}

\begin{proof}
	Let the denominator of $f$ be
	\[\sum_{i=0}^d b_i(x_1,\ldots,x_\ell)t^i,\]
	where $b_i\in\Qp[[x_1,\ldots,x_\ell]]$ is a power series. By hypothesis, if $x\in A$, $b_i(x)=0$ for all $i>0$ and $b_0(x)\ne 0$. This implies that there is a ball $B_x$ centered at $x$ such that, for all $y\in B_x$,
	\begin{equation}\label{eq:ratio}
		\left|\frac{b_i(y)}{b_0(y)}\right|\le (pr)^{-i}.
	\end{equation}
	Let $V$ be the union of all $B_x$. Now \eqref{eq:ratio} holds for all $y\in V$ and we are done by Lemma \ref{lemma:rational}.
\end{proof}

The following lemma is a $p$-adic version of the Tubular Neighborhood Theorem \cite[page 66]{BottTu}.

\begin{definition}[$p$-adic Tubular Neighborhood]\label{def:tubular}
	Let $\ell$ be a positive integer. \letpprime. Let $M$ be an $\ell$-dimensional $p$-adic analytic manifold and $Q$ a compact submanifold of $M$. Let $q$ be the dimension of $Q$. For $k\in\N\cup\{0\}$, let
	\[\mathrm{N}(Q)_k=Q\times (p^k\Zp)^{\ell-q}.\]
	A \textit{$p$-adic tubular neighborhood of $Q$} is a neighborhood $U$ of $Q$ in $M$ such that, for some $k\in\N$, there exists a $p$-adic analytic diffeomorphism $f:\mathrm{N}(Q)_k\to U$ such that $f(q,0)=q$ for all $q\in Q$. See Figure \ref{fig:tubular} for an example.
\end{definition}

\begin{figure}
	\includegraphics{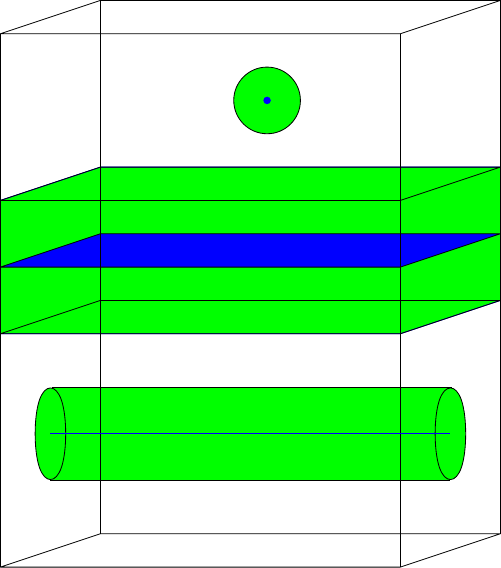}
	\caption{$p$-adic tubular neighborhoods of dimension $0$, $1$ and $2$ submanifolds of a dimension $3$ manifold as in Definition \ref{def:tubular}.}
	\label{fig:tubular}
\end{figure}

\begin{lemma}[$p$-adic analytic Tubular Neighborhood Theorem]\label{lemma:tubular}
	 Let $\ell$ be a positive integer. \letpprime. Let $M$ be an $\ell$-dimensional $p$-adic analytic manifold and let $Q$ be a compact $p$-adic analytic submanifold of $M$. Then there exists a $p$-adic analytic tubular neighborhood of $Q$ in $M$.
\end{lemma}

\begin{proof}
	Since $Q$ is a compact submanifold of the $p$-adic analytic manifold $M$, it is possible to find a finite number of disjoint balls $U_1,\ldots,U_t$ which cover $Q$ and such that there are coordinates $(x_1,\ldots,x_\ell)$ on each $U_i$ which satisfy
	\[U_i\cap Q=\{x\in U_i:x_{q+1}=\ldots=x_\ell=0\}.\]
	That is, there exist balls $B_1,\ldots,B_t\subset(\Qp)^\ell$ and diffeomorphisms $f_i:B_i\to U_i$ such that, if $A_i$ is the projection of $B_i$ to the plane where the last $\ell-q$ coordinates are zero, \[f_i(A_i\times\{(0,\ldots,0)\})=U_i\cap Q.\]
	
	Since the number of balls is finite, there exists some $k\in\N$ such that, for all $i\in\{1,\ldots,t\}$, \[A_i\times (p^k\Zp)^{\ell-q}\subset B_i.\] Now $f_i$ is a diffeomorphism from $A_i\times (p^k\Zp)^{\ell-q}$ to a subset $V_i$ of $U_i$, and the subsets $V_i$ cover $Q$, hence $\bigcup_{i=1}^t V_i$ is the tubular neighborhood we want.
\end{proof}

\begin{lemma}[$p$-adic closed forms versus $p$-adic exact forms]\label{lemma:exact}
	Let $\ell$ and $k$ be integers with $0\le k\le \ell$ and $\ell\ge 1$. \letpprime. Let $M$ be an $\ell$-dimensional $p$-adic analytic manifold and let $Q$ be a compact submanifold of $M$. Let $\alpha$ be a closed $p$-adic analytic $k$-form on $M$ such that
	\[\alpha_m=0\quad\forall m\in Q.\]
	Then, there is a $p$-adic tubular neighborhood $U$ of $Q$ and a $p$-adic analytic $(k-1)$-form $\beta$ on $U$ such that
	\begin{itemize}
		\item $\alpha=\dd\beta$, that is, $\alpha$ is an exact $p$-adic analytic form when restricted to $U$;
		\item $\beta_m=0$ for all $m\in Q$.
	\end{itemize}
\end{lemma}

\begin{proof}
	By Lemma \ref{lemma:tubular}, there exists a $p$-adic analytic tubular neighborhood $V$ of $Q$ as in Definition \ref{def:tubular}. This means that there exists a $p$-adic analytic diffeomorphism $f:\mathrm{N}(Q)_k\to V$ which sends $Q\times \{0\}$ to $Q$, where
	\[\mathrm{N}(Q)_k=Q\times (p^k\Zp)^{\ell-q}.\]
	
	For any $t\in\Zp$, we define $\psi_t:V\to V$ by \[\psi_t(f(x,v))=f(x,tv).\] This is well defined because $f$ is a $p$-adic analytic diffeomorphism, and by definition is a $p$-adic analytic family of maps. Note that $\psi_0(x)\in Q$, $\psi_1(x)=x$, and if $x\in Q$, then $\psi_t(x)=x$ for all $t$. Now we have, by equations \eqref{eq:derivative-flow-constant-omega} and \eqref{eq:cartan}, that
	\begin{equation}\label{eq:exact-step1}
		\frac{\dd}{\dd t}\psi_t^*\alpha=\psi_t^*\lie_{X_t}\alpha=\psi_t^*(\imath_{X_t}\dd\alpha+\dd\imath_{X_t}\alpha)=\psi_t^*\dd\imath_{X_t}\alpha =\dd\psi_t^*\imath_{X_t}\alpha,
	\end{equation}
	where $X_t$ is the vector field given by $\dd\psi_t/\dd t=X_t\circ\psi_t$. Also, since $\psi_t(x)=x$ for all $x\in Q$, $X_t=0$ in $Q$. Let $\beta_t=\psi_t^*\imath_{X_t}\alpha$. Then $\beta_t=0$ in $Q$ and equation \eqref{eq:exact-step1} implies that
	\[\frac{\dd}{\dd t}\psi_t^*\alpha=\dd\beta_t.\]
	Since $\beta_t$ is a $p$-adic analytic family of $(k-1)$-forms, there exists another $p$-adic analytic family of $(k-1)$-forms $\gamma_t$ such that $\gamma_0=0$ and $\frac{\dd}{\dd t}\gamma_t=\beta_t$, and $\frac{\dd}{\dd t}\psi_t^*\alpha=\frac{\dd}{\dd t}\dd\gamma_t$. Since $\beta_t=0$ in $Q$, we also have that $\gamma_t=0$ in $Q$. Since both sides are analytic on $t$, and for $t=0$ we have that $\psi_0^*\alpha=0=\dd\gamma_0$ (because $\psi_0(x)\in Q$ and $\alpha$ vanishes in $Q$), for some small $t_0$ we have that
	\[\psi_{t_0}^*\alpha=\dd\gamma_{t_0}.\]
	Let $U\subset V$ be the image of $\psi_{t_0}$. Then
	\[\alpha|_U=\dd((\psi_{t_0}^{-1})^*\gamma_{t_0}).\]
	Also, since $\gamma_{t_0}$ vanishes at $Q$ and $\psi_{t_0}$ fixes $Q$, $(\psi_{t_0}^{-1})^*\gamma_{t_0}$ also vanishes at $Q$, so this is the form $\beta$ we want.
\end{proof}

The last two results in this section concern symplectomorphisms between unions of balls.

\begin{lemma}\label{lemma:balls}
	\letpprime. Let $r$ be a $p$-adic absolute value. The $2n$-dimensional $p$-adic ball of radius $pr$ is a disjoint union of $p^{2n}$ $2n$-dimensional $p$-adic balls of radius $r$.
\end{lemma}

\begin{proof}
	If $r=p^{-k}$ and $m$ is a point in $\ball(0,pr)=\ball(0,p^{1-k})$, then the digits at the right of the position $k-1$ in all coordinates of $m$ are $0$. We can divide the ball in $p^{2n}$ balls according to the digit at position $k$ of each one of the $2n$ coordinates; these balls are a translation of $\ball(0,p^{-k})$, as we wanted.
\end{proof}

\begin{lemma}\label{lemma:balls2}
	\letpprime. Let $B,B_1,\ldots,B_k$ be $2n$-dimensional $p$-adic balls in $(\Qp)^{2n}$ such that $\vol(B)=\sum_{i=1}^k\vol(B_i)$, where $\vol$ denotes $p$-adic volume. Then $B$ is symplectomorphic to $\bigcup_{i=1}^k B_i$.
\end{lemma}

\begin{proof}
	Let $v$ be the smallest volume among the balls $B_i$. If $\vol(B)=v$, there is nothing to prove. Otherwise, $\vol(B)$ is a multiple of $p^{2n}v$ and the same happens with all volumes greater than $v$. Hence, there are at least $p^{2n}$ balls with volume $v$. Using Lemma \ref{lemma:balls}, we can unify these $p^{2n}$ balls in one with the immediately higher radius. Now we repeat the operation until there is only one ball.
\end{proof}

\section{Generalization of Moser's Path Method to $p$-adic geometry: proof of Theorem \ref{thm:moser}}\label{sec:moser}

Now we use one of the previous results to prove the $p$-adic analytic Moser's Path Method (Theorem \ref{thm:moser}). This method requires additional hypotheses when compared to the real case: not only we need the vector fields $X_{m,t}$ to vanish at the compact submanifold, but also the partial derivatives of their components at $m$. Furthermore, we are restricting to the case when $X_{m,t}$ is given by a power series.

\subsection*{Proof of Theorem \ref{thm:moser}}

We consider the flow of a point $m'$ in $V_m$, for $m\in Q$, by the vector field $X_{m,t}$. In terms of local coordinates centered at $m$, this implies solving an initial value of the form \eqref{eq:initial}, where:
\begin{itemize}
	\item the variable $x$ is $t$,
	\item the variables $y_1,\ldots,y_\ell$ are the local coordinates $x_1^{i_m},\ldots,x_\ell^{i_m}$,
	\item $x_0=0$,
	\item $U_1=\Zp$,
	\item $U_2=p^{-d}\Zp$,
	\item $v=0$,
	\item the initial values are the coordinates of $m'$, and
	\item the series $f_i$, $i\in\{1,\ldots,\ell\}$, correspond to the components of $X_{m,t}$.
\end{itemize}
The series $f_i$, $i\in\{1,\ldots,\ell\}$, satisfy the conditions of Lemma \ref{lemma:initial}, hence there exists an open neighborhood $U_m$ of $m$ such that $U_m\subset V_m$ and, if $m'\in U_m$, it has a unique flow determined by a power series in $t,x_1^{i_m},\ldots,x_\ell^{i_m}$ which converges in $\Zp\times U_m$. Let $\psi_{m,t}:U_m\to U_m$ be this flow. Since $X_{m,t}(m')$ is zero for all $m'\in Q$ and all $t$, $\psi_{m,t}(m')$ is constant at $m'$.

Then (see \eqref{eq:derivative-flow}) we have that
\[\frac{\dd}{\dd t}\psi_{m,t}^*\omega_t =\psi_{m,t}^*\lie_{X_{m,t}}\omega_t+\psi_{m,t}^*\frac{\dd}{\dd t}\omega_t=\psi_{m,t}^*\left(\frac{\dd}{\dd t}\omega_t+\lie_{X_{m,t}}\omega_t\right)=0.\]
The components of $\psi_{m,t}^*\omega_t$ are power series in $t$. If their derivative is $0$, then they must be constant. Since for $t=0$ this equals $\omega_0$, this must hold for all $t\in\Zp$, and $\psi_{m,t}^*\omega_t=\omega_0$.

We now check that the tangent map of $\psi_{m,t}$ at $m$ is independent of $t$:
\begin{align*}
	\frac{\dd}{\dd t}\mathrm{T}_m\psi_{m,t} & =\mathrm{T}_m\frac{\dd}{\dd t}\psi_{m,t} \\
	& =\mathrm{T}_m(X_{m,t}\circ\psi_{m,t}) \\
	& =\mathrm{T}_{\psi_{m,t}(m)} X_{m,t}\circ\mathrm{T}_m\psi_{m,t} \\
	& =\mathrm{T}_m X_{m,t}\circ\mathrm{T}_m\psi_{m,t} \\
	& =0.
\end{align*}
The last line is zero because of the condition (iii) on $X_{m,t}$.

Hence $\mathrm{T}_m\psi_{m,t}$ is constant; since it is the identity at $t=0$, it must be the identity for all $t$. Since $\psi_{m,t}$ is given by a power series centered at $m$, by the inverse function theorem for $p$-adic power series \cite[Proposition 5.9]{Schneider} there exists $U_m'\subset U_m$ such that $\psi_{m,t}$ is injective in $U_m'$. We may assume that $U_m'$ is a ball centered at $m$; then, by Lemma \ref{lemma:initial}, $\psi_{m,t}(U_m')$ is contained in $U_m'$.

At this point we have a family of balls $\{U_m'\}_{m\in Q}$ which cover $Q$. Since $Q$ is a paracompact submanifold of $M$, which is an ultrametric space, we can extract a covering $\{U_m'\}_{m\in A}$ by disjoint balls where $A\subset Q$.

Let \[U=\bigcup_{m\in A} U_m'.\] We define $\psi_t:U\to U$ as follows. Let $x\in U$. There exists a single $m\in A$ such that $x\in U_m'$. Then we take $\psi_t(x)=\psi_{m,t}(x)$.

It is left to show that $\psi_t$ is injective. Suppose that $\psi_t(x_1)=\psi_t(x_2)$ for some $x_1,x_2\in U$. Then $\psi_{m_1,t}(x_1)=\psi_{m_2,t}(x_2)$ for some $m_1,m_2\in A$ such that $x_1\in U_{m_1}'$ and $x_2\in U_{m_2}'$. Since $\psi_{m,t}(U_m')$ is contained in $U_m'$ for all $m\in Q$,
\[\psi_{m_1,t}(x_1)=\psi_{m_2,t}(x_2)\in U_{m_1}'\cap U_{m_2}'\quad\Rightarrow\quad U_{m_1}'\cap U_{m_2}'\ne\varnothing\quad\Rightarrow\quad m_1=m_2.\]
Using that $\psi_{m_1,t}$ is injective in $U_{m_1}'$, this implies $x_1=x_2$, as we wanted.

\begin{figure}
	\includegraphics{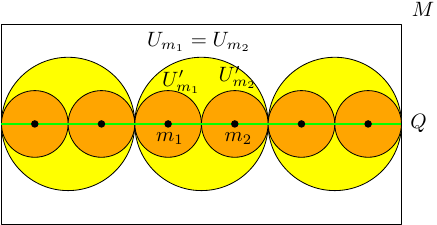}
	\caption{Illustration of the proof of Theorem \ref{thm:moser}. $Q$ is the green line. The balls $U_m$ for several $m$ are represented as yellow balls, and the balls $U_m'$ as orange balls. A set of points $m\in Q$ has been selected such that their balls $U_m'$ cover $Q$, and the union of these $U_m'$ forms the desired set $U$.}
\end{figure}

\section{The Darboux theorem on $p$-adic analytic symplectic manifolds: proof of Theorem \ref{thm:darboux2}}\label{sec:darboux}

Now we prove Darboux's Theorem in the $p$-adic analytic case. This uses the $p$-adic Moser's Path Method (Theorem \ref{thm:moser}), as well as two of the other results in Section \ref{sec:results}.

\begin{theorem}\label{thm:darboux}
	\letnpos. \letpprime. Let $M$ be a $2n$-dimensional $p$-adic analytic manifold. Let $\omega_0$ and $\omega_1$ be two $p$-adic analytic symplectic forms on $M$. Let $m\in M$. Then there exist open neighborhoods $U_0,U_1$ of $m$ and a diffeomorphism $\psi:U_0\to U_1$ such that $\psi^*(\omega_1|_{U_1})=\omega_0|_{U_0}$.
\end{theorem}

\begin{proof}
	\textit{Step 1: find region where the difference is exact.} By \cite[Theorem A.1]{CrePel-williamson1}, any two $p$-adic linear symplectic forms are always linearly symplectomorphic, hence we may assume that \[(\omega_0)_m=(\omega_1)_m.\]
	
	Let \[\alpha=\omega_1-\omega_0.\] Then $\alpha$ is a closed $2$-form and $\alpha_m=0$. By Lemma \ref{lemma:exact}, there exists an open subset $U\subset M$ such that $m\in U$ and $\alpha=\dd\beta$ for some $p$-adic analytic $1$-form $\beta$ in $U$, with $\beta_m=0$.
	
	\textit{Step 2: find region where the forms are globally analytic.} Let $(x_1,\ldots,x_{2n})$ be local coordinates around $m$. Since $\omega_0$ and $\beta$ are analytic, there exists an open subset $U_2\subset U$ such that $m\in U_2$ and the components of $\omega_0$ and $\beta$ are given by power series in the local coordinates $(x_1,\ldots,x_{2n})$. This implies the same for $\alpha$, because $\alpha=\dd\beta$.
	
	\textit{Step 3: decompose $\beta$.} We now write \[\beta=\beta_1+\beta_2,\] where the power series of $\beta_1$ contains only terms of degree $1$, and that of $\beta_2$ contains only terms of higher degree. Hence the power series of $\dd\beta_2$ (in the same coordinates) contains only non-constant terms, and $(\dd\beta_2)_m=0$. Hence
	\[(\dd\beta_1)_m =(\dd\beta_1+\dd\beta_2)_m=(\dd\beta)_m=\alpha_m=0.\]
	But $\dd\beta_1$ must be a constant, so $\dd\beta_1=0$, and $\dd\beta_2=\dd\beta=\alpha$.
	
	\textit{Step 4: shrink the region.} Let \[\Omega_0,A\in\M_{2n}(\Qp[[x_1,\ldots,x_{2n}]])\] be matrices such that
	\[(\omega_0)_x=\sum_{i=1}^{2n}\sum_{j=1}^{2n}(\Omega_0)_{ij}(x)\,\dd x_i\wedge\dd x_j\]
	and
	\[\alpha_x=\sum_{i=1}^{2n}\sum_{j=1}^{2n}A_{ij}(x)\,\dd x_i\wedge\dd x_j.\]
	Let $d=2$ if $p=2$ and otherwise $d=1$. Let \[b\in(\Qp[[x_1,\ldots,x_{2n}]])^{2n}\] be the vector of components of $\beta_2$ in local coordinates. We now apply Lemma \ref{lemma:rational2} to the rational functions
	\[X_i(t,x)=-\Big(\Omega_0(x)+tA(x)\Big)^{-1}b_i(x)\in\Qp[[x_1,\ldots,x_{2n}]](t)\]
	with $U=U_2$ and the subset $A=\{m\}$. We can do that because the denominator evaluated at $m$ is
	\[\det\Big(\Omega_0(m)+tA(m)\Big)=\det(\Omega_0(m))\ne 0.\]
	Hence, there exists an open subset $U_3\subset U_2$ such that $m\in U_3$ and, for every $x\in U_3$ and $i\in\{1,\ldots,2n\}$, the power series of $X_i(t,x)$ as a function of $t$ converges in $\ball(0,p^d)=p^{-d}\Zp$.
	
	\textit{Step 5: define the vector field.} For $t\in T$, let $X_t$ be the vector field on $U_3$ such that $X_t(x)_i=X_i(t,x)$, that is,
	\begin{equation}\label{eq:X}
		X_t(x)=-\Big(\Omega_0(x)+tA(x)\Big)^{-1}b(x).
	\end{equation}
	Then \eqref{eq:X} gives a $p$-adic analytic family of vector fields, and it satisfies that
	$\imath_{X_t}(\omega_0+t\alpha)=-\beta_2$
	for all $t\in T$. Hence, by \eqref{eq:cartan},
	\[\frac{\dd}{\dd t}(\omega_0+t\alpha)+\lie_{X_t}(\omega_0+t\alpha)=\alpha+\imath_{X_t}\dd(\omega_0+t\alpha)+\dd\imath_{X_t}(\omega_0+t\alpha)=\alpha+0-\dd\beta_2=0.\]
	
	\textit{Step 6: apply Moser's trick.} Now we check the hypothesis of Theorem \ref{thm:moser} with $Q=\{m\}$ and $X_{m,t}=X_t$:
	\begin{itemize}
		\item $Q$ is a point, which is paracompact, and there is a single ball $B_1=U_3$.
		\item Because of Step 4, the power series of $X_t$ converges for $t\in p^{-d}\Zp$ and $x\in V_m=U_3$.
		\item The power series $b(x)$ is that of $\beta_2$, which has no terms of degree $0$ or $1$. By equation \eqref{eq:X}, the same holds for $X_t(x)$ in the variables of $x$.
	\end{itemize}
	Thus we can apply Theorem \ref{thm:moser} and obtain an open subset $U_0\subset U_3$ and $\psi_t:U_0\to U_0$ $p$-adic analytic and injective such that $\psi_t^*(\omega_0+t\alpha)=\omega_0|_{U_0}$. This in particular holds for $t=1$, and $\psi_1^*\omega_1=\psi_1^*(\omega_0+\alpha)=\omega_0|_{U_0}$. We conclude the proof by taking $U_1=\psi_1(U_0)$ and $\psi=\psi_1$. See Figure \ref{fig:darboux-proof} for an illustration of the proof.
\end{proof}

\begin{figure}
	\includegraphics{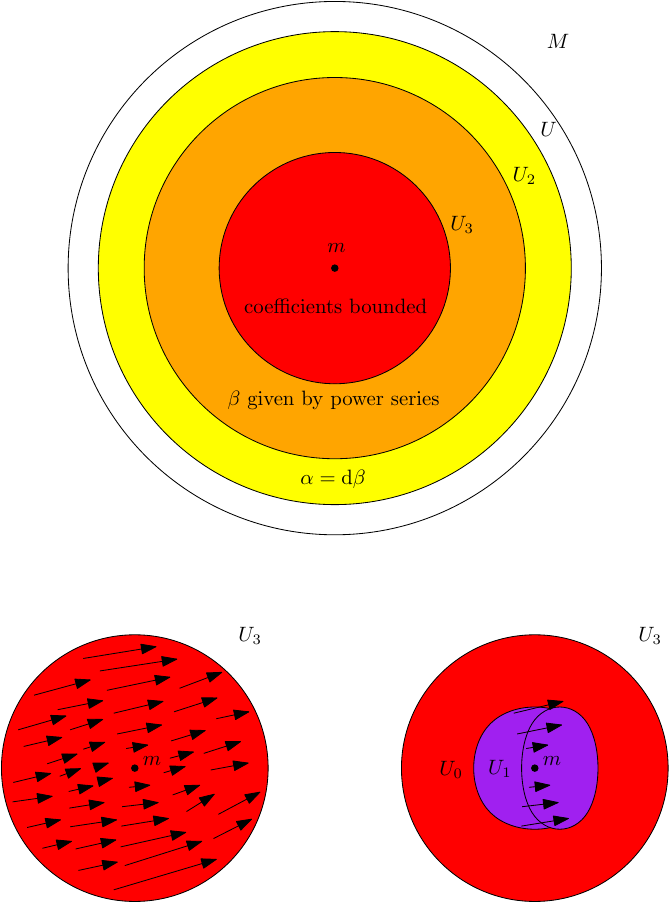}
	\caption{Illustration of the proof of Theorem \ref{thm:darboux}. First figure: the form $\alpha=\omega_1-\omega_0$ is exact on $U$, that is, it is $\dd\beta$ for some $1$-form $\beta$, this form is given by a power series in a subset $U_2\subset U$, whose coefficients are bounded in a subset $U_3\subset U_2$. Second figure: the vector field $X_t$ on $U_3$. Third figure: the resulting sets $U_0$ and $U_1$.}
	\label{fig:darboux-proof}
\end{figure}

\subsection*{Proof of Theorem \ref{thm:darboux2}}
This follows from Theorem \ref{thm:darboux} taking $\omega_0=\sum_{i=1}^n\dd x_i\wedge\dd y_i$.

\section{The Darboux-Weinstein theorem on $p$-adic analytic symplectic manifolds: proof of Theorem \ref{thm:darboux-weinstein}}\label{sec:darbwein}

Now we prove another classical result about the local behavior of symplectic forms. This is similar to Darboux's Theorem, but the starting point is given by two symplectic forms coinciding on a compact submanifold. The idea of the proof is as follows: like in the real case, we want to apply Moser's Path Method, but the $p$-adic equivalent of this result (Theorem \ref{thm:moser}) has additional hypotheses: not only it requires that the vector field vanishes at the compact submanifold, but we also need its derivatives to vanish, and it must be given by a power series.

\subsection*{Proof of Theorem \ref{thm:darboux-weinstein}}

The proof we carry out next is a generalization of the proof of Theorem \ref{thm:darboux}.

\textit{Step 1: find tubular neighborhood.} Let \[\alpha=\omega_1-\omega_0.\] Then $\alpha$ is a closed $2$-form and $\alpha_m=0$ for all $m\in Q$. By Lemma \ref{lemma:exact}, there exists a tubular neighborhood $U$ of $Q$ such that $\alpha=\dd\beta$ for some $p$-adic analytic $1$-form $\beta$ in $U$ with $\beta_m=0$ for all $m\in Q$.

Let $q$ be the dimension of $Q$. Since $U$ is a tubular neighborhood of $Q$, each point in $U$ can be identified with a point in $\mathrm{N}(Q)_k$ for some $k\in\N$, which is $(x,v)$ where $x\in Q$ and $v\in(p^k\Zp)^{2n-q}.$ Points in $Q$ have the form $(x,0)$. Let $B_1,\ldots,B_k$ be balls contained in $U$ which cover $Q$ (these exist because $Q$ is compact) and take local coordinates $x_1^i,\ldots,x_\ell^i$ in each ball.

\textit{Step 2: find the vector fields.} For each $m\in Q$ and $t\in p^{-d}\Zp$, we apply Steps 2 to 5 of the proof of Theorem \ref{thm:darboux} to obtain vector fields $X_{m,t}$ on $U$ such that
\begin{itemize}
	\item for every $m\in Q$ and $t$,
	\[\frac{\dd}{\dd t}(\omega_0+t\alpha)+\lie_{X_{m,t}}(\omega_0+t\alpha)=0;\]
	\item if $m\in B_i$ for some $i$, the power series of $X_{m,t}(x)$ in the local coordinates $x_1^i,\ldots,x_\ell^i$ converges for $x\in V_m$ and $t\in p^{-d}\Zp$, where $V_m\subset B_i$ is a neighborhood of $m$ (the $U_3$ in Step 4 of that proof);
	\item for every $m,m'\in Q$ and $t$, $X_{m,t}(m')=0$ (this happens because $\beta_{m'}=0$);
	\item the partial derivatives of the components of $X_{m,t}$ at $m$ are zero.
\end{itemize}

\textit{Step 3: apply Moser's trick.} Now we can apply Theorem \ref{thm:moser} with the submanifold $Q$ and the vector fields $X_{m,t}$ defined on $V_m$, and obtain an open subset $U_0\subset U$ and $\psi_t:U_0\to U_0$ analytic and injective which fixes $Q$ and such that \[\psi_t^*(\omega_0+t\alpha)=\omega_0|_{U_0}.\] This in particular holds for $t=1$, and \[\psi_1^*\omega_1=\psi_1^*(\omega_0+\alpha)=\omega_0|_{U_0},\]
and the proof is complete by taking $U_1=\psi_1(U_0)$ and $\psi=\psi_1$.

\section{Global classification of $p$-adic analytic symplectic manifolds: proof of Theorems \ref{thm:volumes} and \ref{thm:global-coord}}\label{sec:classification}

In this section we prove our classification of $p$-adic analytic symplectic manifolds. We start with Theorem \ref{thm:volumes} about the possible classes.

\subsection*{Proof of Theorem \ref{thm:volumes}}

\begin{enumerate}
	\item It is enough to show that $\vol(\M_{d,v})=v$. This happens because
	\[\vol(\M_{d,v})=\vol\left(\bigcup_{k=-\infty}^{m}\bigcup_{i=1}^{a_k}\ball_{k,i}\right)=\sum_{k=-\infty}^m\sum_{i=1}^{a_k}p^{dk}=\sum_{k=-\infty}^m a_kp^{dk}=v.\]
	\item It is enough to show that $\vol(\M_{d,v}')=v$. This happens because
	\[\vol(\M_{d,v}')=\vol\left(\bigcup_{k=\ell}^{m}\bigcup_{i=1}^{a_k}\ball_{k,i}\right)=\sum_{k=\ell}^m\sum_{i=1}^{a_k}p^{dk}=\sum_{k=\ell}^m a_kp^{dk}=v.\qedhere\]
\end{enumerate}

\begin{figure}
	\includegraphics{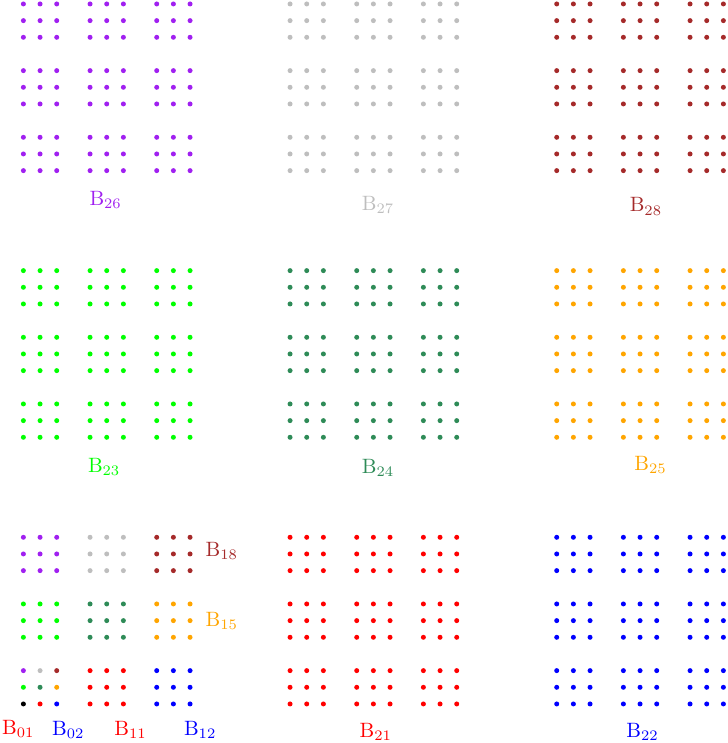}
	\caption{The balls $\ball_{k,i}$, for $p=3$, $d=2$, $0\le k\le 2$, and $1\le i\le 8$. The ball $\ball_{k,i}$ has radius $p^k$ and a color which depends on $i$.}
	\label{fig:balls-ref}
\end{figure}

In order to prove Theorem \ref{thm:global-coord}, the starting point is to use Theorem \ref{thm:darboux2} to write a $p$-adic analytic symplectic manifold as a disjoint union of balls.

\begin{lemma}\label{lemma:disjoint-union}
	\letpprime. Let $(M,\omega)$ be a second countable $p$-adic analytic symplectic manifold. Then $(M,\omega)$ is symplectomorphic to a disjoint union of $p$-adic balls, each one endowed with the standard symplectic form $\omegastd$.
\end{lemma}

\begin{proof}
	By Theorem \ref{thm:darboux2}, for each point $m\in M$, there exists an open subset $U_m$ of $M$ such that $m\in M$ and $U_m$ is symplectomorphic to an open subset $V_m\subset(\Qp)^{2n}$ with the standard symplectic form. Since $M$ is a second countable $p$-adic manifold, it is strictly paracompact by \cite[Proposition 8.7]{Schneider}, which means that the covering $\{U_m\}_{m\in M}$ can be refined to one consisting of disjoint sets. Let $\{U'_i\}_{i\in I}$ be this covering. For each $i\in I$, there exists $m$ such that $U'_i\subset U_m$. Let $V'_i$ be the image of $U'_i$ by the symplectomorphism $U_m\to V_m$. We have that $V'_i$ is an open subset of $(\Qp)^{2n}$. Hence, it is a disjoint union of $p$-adic balls, and $U'_i$ is symplectomorphic to that union. Since $M$ is the disjoint union of the $U'_i$, it is also symplectomorphic to a disjoint union of $p$-adic balls.
\end{proof}

This result, together with the last two results in Section \ref{sec:results}, will allow us to prove Theorem \ref{thm:global-coord}.

\subsection*{Proof of Theorem \ref{thm:global-coord}}

In the three cases, by Lemma \ref{lemma:disjoint-union}, we can suppose that $M$ is a disjoint union of $p$-adic balls. Let $v$ be the volume of $M$.

\begin{enumerate}
	\item[(1a)] Since $M$ is noncompact, the union must be infinite, and it is countable because $M$ is second-countable. By Lemma \ref{lemma:balls}, we may assume, without loss of generality, that the radii of the balls is less or equal than $1$ (otherwise, we divide all balls of radius greater than $1$ into smaller balls). We now modify the union so that it contains a finite number of balls of each possible radius. In order to achieve this, if there are infinitely many balls of radius $r$, we leave the first as such, decompose the second into $p^{2n}$ balls of radius $r/p$, the third into $p^{4n}$ balls of radius $r/p^2$, and so on.
	
	Let $B_1,B_2,\ldots$ be the balls in $M$ with radii $r_1,r_2,\ldots$ in decreasing order, that is, with $r_i\ge r_{i+1}$. We define $a_0=0$ and, for each $k\in\N$,
	\[a_k=\sum_{i=1}^k r_i^{2n}.\]
	Since $r_i$ and $r_{i+1}$ are $p$-adic absolute values and $r_i\ge r_{i+1}$, we have $r_i/r_{i+1}\in\Z$. This in turn implies that $r_i/r_j\in\Z$ for all $i,j\in\N$ with $i<j$ and
	\[\frac{a_k}{r_k^{2n}}=\sum_{i=1}^k\left(\frac{r_i}{r_k}\right)^{2n}\in\Z.\]
	We also have that
	\[\lim_{k\to\infty} a_k=\sum_{i=1}^\infty r_i^{2n}=\vol(M)=\infty.\]
	
	We claim that, for all $m\in\N$, there is $k(m)\in\N$ such that $a_{k(m)}=m$. Suppose otherwise that $a_k<m<a_{k+1}$ for some $k\in\N$. Then we have that
	\[\frac{a_k}{r_{k+1}^{2n}}<\frac{m}{r_{k+1}^{2n}}<\frac{a_{k+1}}{r_{k+1}^{2n}}=\frac{a_k}{r_{k+1}^{2n}}+1,\]
	which is impossible because all the terms are integers.
	
	Now the sum of volumes of the balls $B_{k(m)+1},B_{k(m)+2},\ldots,B_{k(m+1)}$ is $a_{k(m+1)}-a_{k(m)}=1$, and by Lemma \ref{lemma:balls2}, their union is symplectomorphic to a ball of radius $1$. Since $\Qp$ is a countable union of balls of radius $1$, the union of all balls is symplectomorphic to $\Qp$ and we are done.
	
	\item[(1b)] Again, since $M$ is noncompact, the union must be infinite, and it is countable because $M$ is second-countable. Let $\ball_{mi}$ be a ball contained in $\M_{2n,v}$ such that $m$ is the maximum possible. We claim that there exists a subfamily of balls contained in $M$ which is symplectomorphic to $\ball_{mi}$.
	
	Let $B_1,B_2,\ldots$ be the balls in $M$ with radii $r_1,r_2,\ldots$ in decreasing order, that is, with $r_i\ge r_{i+1}$. We define $a_0=0$ and, for each $k\in\N$,
	\[a_k=\sum_{i=1}^k r_i^{2n}.\]
	Since $r_i$ and $r_{i+1}$ are $p$-adic absolute values and $r_i\ge r_{i+1}$, we have $r_i/r_{i+1}\in\Z$. This in turn implies that $r_i/r_j\in\Z$ for all $i,j\in\N$ with $i<j$ and
	\[\frac{a_k}{r_k^{2n}}=\sum_{i=1}^k\left(\frac{r_i}{r_k}\right)^{2n}\in\Z.\]
	We also have that
	\[\lim_{k\to\infty} a_k=\sum_{i=1}^\infty r_i^{2n}=\vol(M)=v.\]
	
	Let $k$ be such that $a_{k-1}<p^{2nm}\le a_k$. This $k$ exists since \[\lim_{k\to\infty} a_k=v=\vol(\M_{2n,v})>\vol(\ball_{m,i})=p^{2nm}.\] If $a_k=p^{2nm}$, Lemma \ref{lemma:balls2} means that the union $B_1\cup\ldots\cup B_k$ is symplectomorphic to $\ball_{mi}$. Suppose otherwise that $a_k>p^{2nm}$. Because of our choice of $m$, $p^{2n(m+1)}>v$, hence $r_k^{2n}<p^{2n(m+1)}$ and $r_k<p^{m+1}$, which implies $r_k\le p^m$. Then we have that
	\[\frac{a_{k-1}}{r_k^{2n}}<\frac{p^{2nm}}{r_k^{2n}}<\frac{a_k}{r_k^{2n}}=\frac{a_{k-1}}{r_k^{2n}}+1,\]
	which is impossible because all the terms are integers.
	
	We have obtained so far a symplectic embedding from $\ball_{mi}$ to $M$ whose image is a union of a subfamily of balls. By removing from $M$ these balls and iterating the construction, we can obtain a symplectic embedding $\varphi$ from $\M_{2n,v}$ to $M$ whose image is the union of a subfamily of balls. Suppose that $\varphi$ is not a symplectomorphism. Then there is a ball $B_0$ in $M$ which is not in the image of $\varphi$. This means that
	\[v=\vol(M)\ge\vol(\M_{2n,v})+\vol(B_0)=v+\vol(B_0)>v,\]
	which is a contradiction. Hence, $\varphi$ is a symplectomorphism and we are done.
	
	\item[(1c)] In this case the union is finite. Let $r$ be the smallest radius of all balls in $M$. Since $v$ is a multiple of $r$, by the definition of $\M_{2n,v}'$, all balls in $\M_{2n,v}'$ have at least radius $r$. By Lemma \ref{lemma:balls}, $M$ and $\M_{2n,v}'$ can be written as disjoint unions of balls of radius $r$. Since the two manifolds have the same volume, the number of balls in both unions must be the same, and $M$ and $\M_{2n,v}'$ are symplectomorphic.
	
	\item[(2)] This is a consequence of the fact that all symplectomorphisms preserve $p$-adic volume.
\end{enumerate}

See Figure \ref{fig:balls-ref2} for an example of the proof.

\begin{figure}
	\includegraphics{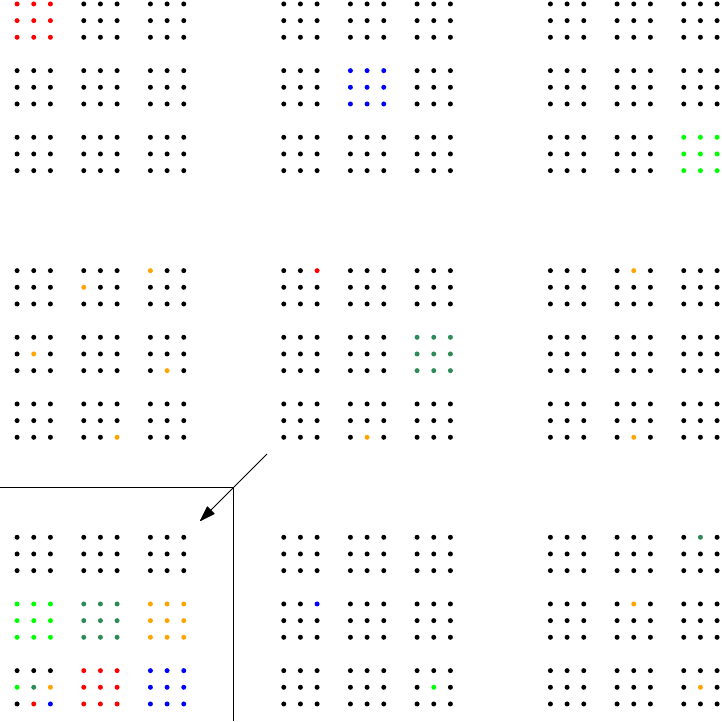}
	\caption{An example of the proof of Theorem \ref{thm:global-coord}. The colored balls form a $p$-adic analytic symplectic manifold in $(\Q_3)^2$. Those with radius $3$ can be mapped to $\ball_{1,1}$, $\ball_{1,2}$, $\ball_{1,3}$ and $\ball_{1,4}$ according to their colors, while those of radius $1$ are mapped to $\ball_{0,1}$, $\ball_{0,2}$, $\ball_{0,3}$, $\ball_{0,4}$, $\ball_{0,5}$, and the remaining nine to $\ball_{1,5}$.}
	\label{fig:balls-ref2}
\end{figure}

\section{Application to nonlinear Schr\"odinger equation}\label{sec:examples}

We will now give an explicit example of the $p$-adic Darboux's Theorem based on the ideas for the real case in \cite{CPP}.

The \textit{Discrete Nonlinear Schr\"odinger Model (dNLS)} is a Hamiltonian dynamical system given by
\[H(\psi)=\sum_{j\in J}\left(\epsilon(\psi_j\bar{\psi}_{j+1}+\bar{\psi}_j\psi_{j+1})+\frac{\gamma}{2}|\psi_j|^4\right)\]
where $J=\{1,\ldots,n\}$ or $J=\Z$, $\epsilon\in\R^+$, $\gamma\in\R$, and $\psi_j\in\C$ for each $j\in J$. The standard symplectic form is given in those coordinates by
\[\omegastd=\ii\sum_{j\in J}^n\dd \psi_j\wedge\dd\bar{\psi}_j\]
and the movement equation is
\[\ii\frac{\dd}{\dd t}\psi_j=\epsilon(\psi_{j+1}+\psi_{j-1})+\gamma|\psi_j|^2\psi_j.\]

A generalization of this model changes the symplectic form to
\[\omega=\ii\sum_{j\in J}\frac{\dd \psi_j\wedge\dd\bar{\psi}_j}{1+\mu|\psi_j|^2}\]
and takes as Hamiltonian
\[H(\psi)=\sum_{j\in J}\left(\epsilon(\psi_j\bar{\psi}_{j+1}+\bar{\psi}_j\psi_{j+1})+\frac{\gamma}{\mu}|\psi_j|^2-\frac{\gamma}{\mu^2}\log(1+\mu|\psi_j|^2)\right)\]
where $J=\{1,\ldots,n\}$ or $J=\Z$, $\epsilon,\mu\in\R^+$, $\gamma\in\R$, and $\psi_j\in\C$ for each $j\in J$. Now the movement equation is
\[\ii\frac{\dd}{\dd t}\psi_j=\epsilon(1+\mu|\psi_j|^2)(\psi_{j+1}+\psi_{j-1})+\gamma|\psi_j|^2\psi_j.\]
When $\gamma=0$, it is known as \textit{Ablowitz-Ladik Model}; when $\gamma\ne 0$, it is the \textit{Salerno Model}. When $\mu$ tends to $0$, we recover the dNLS model. According to \cite{CPP}, after changing coordinates to those given by Darboux's Theorem, the system also becomes the dNLS model.

For our purposes, we will use real coordinates $(x_1,y_1,\ldots,x_n,y_n)$ instead of complex $(\psi_1,\ldots,\psi_n)$:
\[\left\{
\begin{aligned}
	\psi_j & =\frac{x_j+\ii y_j}{\sqrt{2}}; \\
	\bar{\psi}_j & =\frac{x_j-\ii y_j}{\sqrt{2}}.
\end{aligned}
\right.\]
We will apply Darboux's Theorem (Theorem \ref{thm:darboux2}) to the $p$-adic equivalent of this system, where the $2n$ variables are in $\Qp$ instead of in $\R$. Often it is not possible to obtain an explicit expression of the Darboux coordinates because it involves integrating very complicated expressions, but one can follow the proof of the $p$-adic analytic Darboux's Theorem and at least derive non-explicitly the symplectomorphism and some of its properties. This is the case for the example we are treating in this section.

Though we have a symplectic form on $(\Qp)^{2n}$, it is actually the sum of $n$ symplectic forms with the same expression, each one on $(\Qp)^2$, so we can apply Theorem \ref{thm:darboux2} to the first two and the conclusion will apply to all of them:
\[\omega_1(x_1,y_1)=\frac{\omegastd(x_1,y_1)}{1+\nu(x_1^2+y_1^2)},\]
where $\nu=\mu/2$ and $\omegastd(x_1,y_1)=\dd x_1\wedge\dd y_1$. We now want to find local coordinates $(x_0,y_0)$ around $0$ in $(\Qp)^2$ which are symplectic, that is, in which
\[\omega_1(x_1,y_1)=\omegastd(x_0,y_0).\]

\begin{proposition}\label{prop:dNLS}
	The diffeomorphism corresponding to integrating
	\[X_t=\frac{1+\nu(x_1^2+y_1^2)}{1+(1-t)\nu(x_1^2+y_1^2)}\left(1-\frac{\log(1+\nu(x_1^2+y_1^2))}{2\nu(x_1^2+y_1^2)}\right)\begin{pmatrix}
		x_1 \\ y_1
	\end{pmatrix}\]
	takes $\omega_1$ to $\dd x_1\wedge\dd y_1$ (this is the diffeomorphism $\psi$ found theoretically in Theorem \ref{thm:darboux2}).
\end{proposition}

\begin{proof}
	First of all, we need to apply a linear transformation to achieve that $\omega_1=\omegastd$ at the origin. In this case, this transformation is the identity.
	
	The next step is to make
	\[\alpha=\omega_1-\omegastd=-\frac{\nu(x_1^2+y_1^2)}{1+\nu(x_1^2+y_1^2)}\omegastd\]
	and find a $1$-form $\gamma$ such that $\dd\gamma=\alpha$ near the origin. We can write
	\[\gamma=f(\nu(x_1^2+y_1^2))(-y_1\dd x_1+x_1\dd y_1)\]
	so that
	\begin{align*}
		\dd\gamma & =f(\nu(x_1^2+y_1^2))(-\dd y_1\wedge\dd x_1+\dd x_1\wedge\dd y_1) \\
		& \;\;\;+f'(\nu(x_1^2+y_1^2))(-2\nu y_1^2\dd y_1\wedge\dd x_1+2\nu x_1^2\dd x_1\wedge\dd y_1) \\
		& =2f(\nu(x_1^2+y_1^2))\dd x_1\wedge\dd y_1 \\
		& \;\;\;+2f'(\nu(x_1^2+y_1^2))\nu(x_1^2+y_1^2)\dd x_1\wedge\dd y_1,
	\end{align*}
	and we need
	\[2f(t)+2tf'(t)=-\frac{t}{1+t},\]
	that is
	\[2\frac{\dd}{\dd t}(tf(t))=-\frac{t}{1+t},\]
	\[tf(t)=\int-\frac{t}{2(1+t)}\dd t=\frac{1}{2}(\log(1+t)-t),\]
	and we finally obtain
	\[f(t)=\frac{1}{2}\left(\frac{\log(1+t)}{t}-1\right).\]
	Substituting in the expression for $\gamma$,
	\[\gamma=\left(\frac{\log(1+\nu(x_1^2+y_1^2))}{2\nu(x_1^2+y_1^2)}-1\right)(-y_1\dd x_1+x_1\dd y_1).\]
	
	The next step is to obtain the vector field $X_t$ such that
	\[\imath_{X_t}(\omegastd+t\alpha)=-\gamma.\]
	This gives
	\[\begin{pmatrix}
		0 & -1+t\frac{\nu(x_1^2+y_1^2)}{1+\nu(x_1^2+y_1^2)} \\
		1-t\frac{\nu(x_1^2+y_1^2)}{1+\nu(x_1^2+y_1^2)} & 0
	\end{pmatrix}X_t
	=\left(1-\frac{\log(1+\nu(x_1^2+y_1^2))}{2\nu(x_1^2+y_1^2)}\right)\begin{pmatrix}
		-y_1 \\ x_1
	\end{pmatrix}\]
	and
	\begin{align*}
		X_t & =\begin{pmatrix}
			0 & \frac{1+\nu(x_1^2+y_1^2)}{1+(1-t)\nu(x_1^2+y_1^2)} \\
			-\frac{1+\nu(x_1^2+y_1^2)}{1+(1-t)\nu(x_1^2+y_1^2)} & 0
		\end{pmatrix}
		\left(1-\frac{\log(1+\nu(x_1^2+y_1^2))}{2\nu(x_1^2+y_1^2)}\right)\begin{pmatrix}
			-y_1 \\ x_1
		\end{pmatrix} \\
		& =\frac{1+\nu(x_1^2+y_1^2)}{1+(1-t)\nu(x_1^2+y_1^2)}\left(1-\frac{\log(1+\nu(x_1^2+y_1^2))}{2\nu(x_1^2+y_1^2)}\right)\begin{pmatrix}
			x_1 \\ y_1
		\end{pmatrix}.
	\end{align*}
	
	Now we have $X_t$, its integral will give us the required coordinate change.
\end{proof}

\begin{remark}
	This vector field can be related to $V_t$ in \cite[Theorem 2.1]{CPP} by $X_t=-V_{1-t}$. This corresponds exactly to the effect of inverting the coordinate change: we are taking $\omega_0$ as the standard form and $\omega_1$ as a non-standard form, while they take $\omega_1$ as the standard form. This means that the \textit{algebraic} expression of the vector field in the real and $p$-adic cases is essentially the same. The differences between the proofs of real and $p$-adic Darboux's theorems are only due to the different \textit{topology} of both fields.
\end{remark}

\section{Final remarks and context of the paper}

\begin{remark}[References in $p$-adic/symplectic geometry and context of the paper]
	Symplectic geometry originates in the study of planetary motions in the seventeenth and eighteenth centuries and has deep connections with mathematical physics, analysis/PDEs, representation theory and many other subjects (see \cite{MarRat,Palmer,TopApp2023,RWZ,Weinstein-symplectic}). For an introduction to symplectic geometry, we recommend the textbooks and surveys \cite{Cannas,HofZeh,MarRat,OrtRat}, and for its connection to the theory of quantum states see \cite{Gosson,WHTH}. The role of $p$-adic numbers in physics has been extensively studied, see for example \cite{BreFre, DKKV, DKKVZ, RTVW, VlaVol}. A construction of the $p$-adic symplectic and Heisenberg groups and the Maslov index can be found in \cite{HuHu,Zelenov}. The present paper is part of a larger program proposed by Voevodsky, Warren and the second author in \cite{PVW} to develop a $p$-adic version of symplectic geometry with the goal of later implementing it using a proof assistant (in the language of homotopy type theory and Voevodsky's Univalent Foundations \cite{APW,PelWar2}). Within this program, \emph{the results of our paper clarify the local theory completely.} In our previous works in $p$-adic symplectic geometry we focused on matrices/integrable systems \cite{CrePel-JC,CrePel-williamson1,CrePel-williamson2} and flexibility/rigidity \cite{CrePel-nonsqueezing}.
\end{remark}

\begin{remark}[Gromov's Nonsqueezing]
	Theorem \ref{thm:global-coord} implies that Gromov's Nonsqueezing Theorem \cite{Gromov} does not hold in general for $p$-adic analytic symplectic manifolds, as we proved explicitly in \cite[Theorem B]{CrePel-nonsqueezing} (see Table \ref{table:comparison}). Indeed, a $p$-adic cylinder is a noncompact second-countable $2n$-dimensional $p$-adic analytic symplectic manifold with infinite volume, hence Theorem \ref{thm:global-coord} implies that it is symplectomorphic to $(\Qp)^{2n}$.
\end{remark}

\begin{remark}[Second-countable axiom]
	Being second-countable is sometimes taken as part of the definition of manifold, see for example \cite{Lee}. This is often the case of real symplectic geometry. In Schneider's book \cite{Schneider}, which concerns $p$-adic analytic manifolds, this is not assumed by default to avoid being too restrictive.
	
	Without the assumption of being second-countable Theorem \ref{thm:global-coord} is false, because, for example, the union of an uncountable number of copies of $(\Qp)^{2n}$ cannot be embedded into $(\Qp)^{2n}$, and the statement of Theorem \ref{thm:classification} makes no sense in general, because the volume of a manifold is not necessarily well defined without this hypothesis. Actually, we can construct families, with greater cardinality than the continuum, of (non-second-countable) $p$-adic analytic manifolds which are not even homeomorphic \cite{GarMoh}. In view of this, we are tempted to say that the ``correct'' formulation of $p$-adic symplectic geometry is without assuming second-countability, unlike it is often done in the real case.
\end{remark}

\begin{remark}[Darboux-Weinstein's Theorem does not follow from Serre's Theorem]
	Theorem \ref{thm:darboux-weinstein} is not a consequence of Theorem \ref{thm:classification}, because the latter ensures the existence of a symplectomorphism but does not impose any condition on it, while the former specifies that the symplectomorphism fixes the compact submanifold $Q$.
\end{remark}

\begin{table}
	\begin{tabular}{|m{6.4cm}|>{\centering\arraybackslash}m{2cm}|>{\centering\arraybackslash}m{2cm}|} \hline
		& symplectic geometry & $p$-adic symplectic geometry \\\hline
		Darboux's Theorem & \checkmark & \checkmark \\\hline
		Global Darboux's Theorem (Serre's classification) & \texttimes & \checkmark \\ \hline
		Darboux-Weinstein's Theorem & \checkmark & \checkmark \\\hline
		Gromov's linear nonsqueezing & \checkmark & \checkmark \\\hline
		Gromov's nonsqueezing & \checkmark & \texttimes \\\hline
		Equivariant Gromov's nonsqueezing & \checkmark & \checkmark \\ \hline
		Few normal forms for integrable systems in dimension $4$ & \begin{tabular}{c}
			\checkmark \\ (7 forms)
		\end{tabular} & \begin{tabular}{c}
			\texttimes \\ \!\!\!(hundreds)
		\end{tabular}  \\\hline
		Interesting physical models & \checkmark & \checkmark \\ \hline
	\end{tabular}
	\caption{Comparison of real and $p$-adic symplectic geometry. All the classical results of real symplectic geometry included here also hold in $p$-adic symplectic geometry, with the exceptions of Gromov's nonsqueezing and the normal forms for integrable systems. In the first case, this is due to the fact that symplectic transformations, in a global and nonlinear sense, have more freedom in the $p$-adic case than in the real case, though locally and linearly they are similar. For the classification of integrable systems, the reason lies behind the higher algebraic complexity of the $p$-adic numbers. See \cite{CrePel-JC,CrePel-williamson1,CrePel-williamson2,CrePel-nonsqueezing} for the proofs.}
	\label{table:comparison}
\end{table}

\begin{remark}[Applications to integrable systems]
	Theorem \ref{thm:darboux2} can be used to give a classification of critical points in $p$-adic analytic integrable systems. The first step to classify such a critical point is to find Darboux's coordinates around the point, and after that we need to make a linear symplectic change of coordinates which gives new Darboux's coordinates, in which the integrable system has a canonical form. The treatment in dimension $2$ and $4$ is done in \cite{CrePel-williamson1}. The classification, even in those low dimensions, has many more types then its real counterpart: there are $37$, $56$ or $222$ different normal forms for a critical point of a $p$-adic analytic integrable system, depending on $p$, while the real case has only $6$ different types.
\end{remark}

\appendix
\section{$p$-adic numbers and $p$-adic geometry}\label{sec:prelim}

\subsection{$p$-adic numbers, power series and analytic functions}

In this section we explain the concepts we need about the $p$-adic numbers. The following presentation follows mainly Schneider's book on $p$-adic Lie groups \cite{Schneider}.

\letpprime. For $n\in\Z$, we define the \textit{$p$-adic absolute value} $|n|_p$ of $n$ as \[|n|_p=p^{-k},\] where $p^k$ is the highest power of $p$ which divides $n$. For $m/n\in\Q$, we define the $p$-adic absolute value of $m/n$ as
\[\left|\frac{m}{n}\right|_p=\frac{|m|_p}{|n|_p}.\]
We say that $r\in\Q$, with $r\ne 0$, is a \emph{$p$-adic absolute value} if it is in the image of the $p$-adic absolute value function (equivalently, it is a power of $p$). The field $\Qp$ is the metric completion of $\Q$ with respect to the $p$-adic absolute value.

Let $\ell$ be a positive integer. We define the \textit{$p$-adic norm} of $(x_1,\ldots,x_\ell)\in(\Qp)^\ell$ by \[\|(x_1,\ldots,x_n)\|_p=\max\Big\{|x_i|_p:i\in\{1,\ldots,\ell\}\Big\}.\]
Let $x\in(\Qp)^\ell$ and let $r$ be a $p$-adic absolute value. The \textit{ball of center $x$ and radius $r$} in $(\Qp)^\ell$, denoted by $\ball(x,r)$, is the set
\[\Big\{y\in(\Qp)^\ell:\|y-x\|_p\le r\Big\}=\Big\{y\in(\Qp)^\ell:\|y-x\|_p< pr\Big\}.\]
See Figures \ref{fig:ball} and \ref{fig:ball2} for representations of balls for $p=5$ in dimension $2$ and $3$. A $p$-adic ball is an open and closed set.

\begin{figure}
	\includegraphics{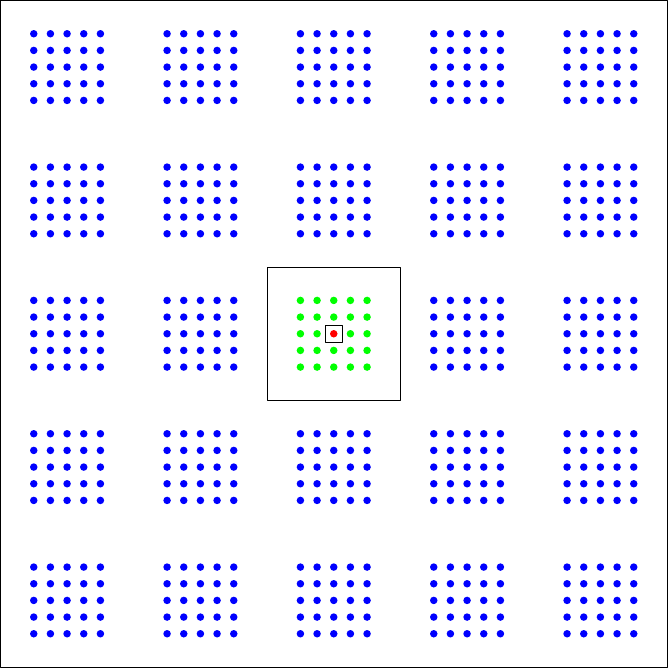}
	\caption{A representation of the ball $\ball(0,25)$ in $(\Q_5)^2$. Each point is a ball of radius $1$. The red point in the center is $\ball(0,1)$ and the green points form $\ball(0,5)$.}
	\label{fig:ball}
\end{figure}

\begin{figure}
	\includegraphics{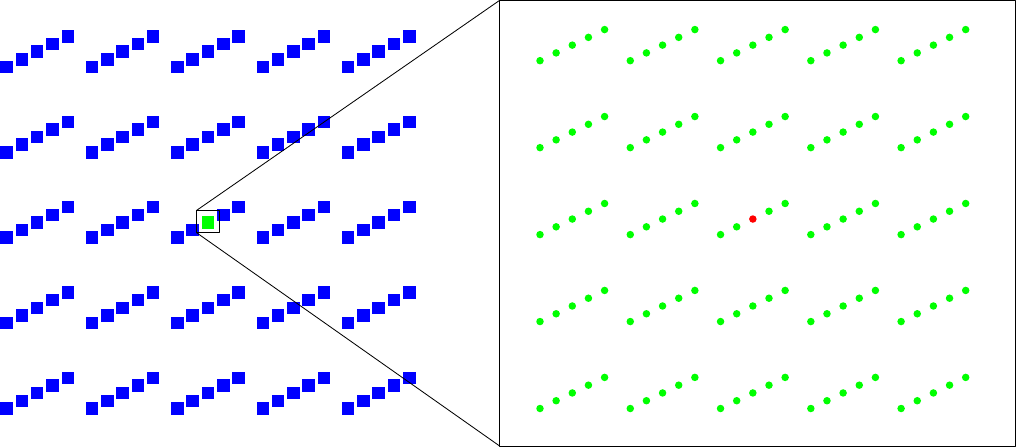}
	\caption{Left: a representation of the ball $\ball(0,25)$ in $(\Q_5)^3$. Each square is a ball of radius $5$. Right: the ball $\ball(0,5)$, where each point is a ball of radius $1$ and the red point in the center is $\ball(0,1)$.}
	\label{fig:ball2}
\end{figure}

A \textit{$p$-adic power series} in $(\Qp)^\ell$ is a series of the form
\[\sum_{I=(i_1,\ldots,i_\ell)\in \N^\ell}a_I(x_1-x_{0,1})^{i_1}\ldots(x_\ell-x_{0,\ell})^{i_\ell},\]
where $(x_{0,1},\ldots,x_{0,\ell})\in(\Qp)^\ell$ and $a_I\in\Qp$.
Given a $p$-adic power series $f\in\Qp[[x_1,\ldots,x_\ell]]$, for each $k\in\N$, let $C_k$ be the maximum of $|a_I|_p$ where $I=(i_1,\ldots,i_\ell)$ and $\sum_{j=1}^\ell i_j=k$. Let $r$ be a $p$-adic absolute value. Then it is easy to check that $f$ converges at all points of $\ball(x_0,r)$ if and only if $\lim_{k\to\infty}C_kr^k=0$. This follows from the fact that a $p$-adic series converges if and only if its terms tend to $0$.

Finally, if $U\subset (\Qp)^{\ell_1}$ and $V\subset(\Qp)^{\ell_2}$ are open sets, a function $f:U\to V$ is \textit{$p$-adic analytic} if $U$ can be expressed as
$U=\bigcup_{i\in I}U_i$
where $U_i=\ball(x_i,r_i)$, for some $x_i\in(\Qp)^{\ell_1}$ and $p$-adic absolute values $r_i$, and there are power series $f_i$ converging in $U_i$ such that $f(x)=f_i(x)$ for every $x\in U_i$.

\subsection{$p$-adic analytic manifolds, vector fields and forms}

Let $\ell$ be a positive integer. \letpprime. Let $M$ be a Hausdorff topological space. An \emph{$\ell$-dimensional $p$-adic analytic atlas} is a set of functions $A=\{\phi:U_\phi\to V_\phi\}$, where $U_\phi\subset M$ and $V_\phi\subset (\Qp)^\ell$ are open subsets, such that:
\begin{itemize}
	\item $\phi:U_\phi\to V_\phi$ is a homeomorphism;
	\item for any $\phi,\psi\in A$, the change of charts
	\[\psi\circ \phi^{-1}:\phi(U_\phi\cap U_\psi)\to \psi(U_\phi\cap U_\psi)\]
	is bi-analytic, i.e. it is analytic with analytic inverse.
\end{itemize}
A topological space $M$ endowed with an $\ell$-dimensional $p$-adic analytic atlas is called a \textit{$p$-adic analytic manifold.}
The integer $\ell$ is called the \emph{dimension of $M$.}

Let $M$ and $N$ be $p$-adic analytic manifolds of dimensions $\ell_1$ and $\ell_2$ respectively. A map $F:M\to N$ is \textit{$p$-adic analytic} if, for any $m\in M$, there exist neighborhoods $U_\phi$ of $m$ and $U_\psi$ of $F(m)$ such that $\psi\circ F\circ\phi^{-1}$ is $p$-adic analytic (this composition is a function from a subset of $(\Qp)^{\ell_1}$ to a subset of $(\Qp)^{\ell_2}$).

The following concepts are analogous to the real case:

\begin{itemize}
	\item Given a $p$-adic analytic manifold $M$, a function $f:M\to\Qp$ is \textit{$p$-adic analytic} if it is analytic as a map between manifolds. Equivalently, for the charts $\phi$ of $M$, $f|_{U_\phi}\circ \phi^{-1}$ is analytic on $\phi(U_\phi)$. The space of analytic maps $M\to\Qp$ is denoted by $\Omega^0(M)$.
	
	\item A \textit{tangent vector} to $m\in M$ is a linear map $v:\Omega^0(M)\to\Qp$ such that
	$v(fg)=v(f)g(m)+f(m)v(g)$
	for all $f,g\in\Omega^0(M)$. We denote the space of tangent vectors to $m$ by $\mathrm{T}_mM$ and by $\mathrm{T}M=\bigcup_{m\in M}\mathrm{T}_mM$ the $p$-adic tangent bundle, which is again defined by analogy with the real case. It has the structure of a $p$-adic analytic manifold.
	
	\item The concept of \emph{$p$-adic analytic vector field} and \emph{$p$-adic analytic $k$-form} are defined by analogy with the definitions in the real case (see \cite[Chapter 9]{Lee}), where the functions involved in the definitions are $p$-adic analytic in the sense above. Essentially, a vector field is an analytic section of the tangent bundle and an analytic $k$-form is an analytic section of $\amalg_{m\in M}\Lambda^k(\mathrm{T}_mM)$, where $\amalg$ denotes disjoint union and $\Lambda^k(V)$ denotes the space of alternating $k$-tensors over a vector space $V$.
	
	\item A \emph{$p$-adic analytic (embedded) submanifold} $Q$ of $M$ is a subset of $M$ for each point
	$m \in Q$ there exists a chart $\phi:U\to V$ such that $m \in U$ and $\phi(U \cap Q)$
	is a $k$-slice of $\phi(U)$, that is, an open set of the form
	\[\{(x_1,\ldots,x_k, x_{k+1},\ldots, x_n) \in \phi(U) : x_{k+1} = c_{k+1},\ldots, x_n = c_n\},\]
	for some constants $c_{k+1},\ldots,c_n\in\Qp$.
	
	\item The \textit{pullback} $F^*(\alpha)\in\Omega^k(M)$ of a form $\alpha\in\Omega^k(N)$ by $F$, the \textit{push-forward} $F_*(v)\in\mathrm{T}_mN$ of a vector $v\in \mathrm{T}_mM$, and, if $F$ is bi-analytic, the \textit{push-forward} $F_*(X)\in\X(N)$ of a vector field $X\in \X(M)$, are defined analogously to the real case.
	
	\item Similarly for the \textit{wedge operator}, the \textit{differential operator} $\dd$, the \textit{contraction operator} $\imath_X$ and the \textit{Lie derivative} $\lie_X$ for a vector field $X$. Like in the real case, \emph{Cartan's Magic Formula} holds in the $p$-adic case:
	\begin{equation}\label{eq:cartan}
		\lie_X\omega=\imath_X\dd\omega+\dd\imath_X\omega.
	\end{equation}
\end{itemize}

\letnpos. \letpprime. Following Pelayo-Voevodsky-Warren \cite[Section 7.2]{PVW}, a \textit{$2n$-dimensional $p$-adic (analytic) symplectic manifold} is a pair $(M,\omega)$, where $M$ is a $2n$-dimensional $p$-adic analytic manifold and $\omega$ is a closed non-degenerate $p$-adic analytic $2$-form, that is, $\dd\omega=0$ and for all $m\in M$ and $u\in\mathrm{T}_mM,u\ne 0,$ there is $v\in\mathrm{T}_mM$ such that $\omega(u,v)\ne 0$. The form $\omega$ is called a \emph{$p$-adic analytic symplectic form} or simply a \emph{$p$-adic symplectic form.}

\letpprime. Let $A$, $B$ and $T$ be $p$-adic analytic manifolds. Let $f_t:A\to B$ be a $p$-adic analytic map for each $t\in T$. We say that $\{f_t\}_{t\in T}$ is a \emph{$p$-adic analytic family of maps} if
\begin{align*}
	T\times A & \to B \\
	(t,a) & \mapsto f_t(a)
\end{align*}
is a $p$-adic analytic map. In addition, we say that $\{f_t\}_{t\in T}$ is a \emph{$p$-adic analytic family of embeddings} if each $f_t$ is an embedding.
	
Let $X_t\in\X(A)$ be a vector field for each $t\in T$. We say that $\{X_t\}_{t\in T}$ is a \emph{$p$-adic analytic family of vector fields} if, for all $f\in\Omega_0(A)$, $\{X_t(f)\}_{t\in T}$ is a $p$-adic analytic family of maps.
	
Let $\omega_t\in\Omega_k(A)$ be a $k$-form for each $t\in T$. We say that $\{\omega_t\}_{t\in T}$ is a \emph{$p$-adic analytic family of $k$-forms} if, for every $X_1,\ldots,X_k\in\X(A)$, $\{\omega_t(X_1,\ldots,X_k)\}_{t\in T}$ is a $p$-adic analytic family of maps.
	
If $\psi_t$ is the flow of $X_t$ and $\{\omega_t\}_{t\in p^{-d}\Zp}$ is a $p$-adic analytic family of $k$-forms on $M$, then
\begin{equation}\label{eq:derivative-flow}
	\frac{\dd}{\dd t}\psi_t^*\omega_t=\psi_t^*\left(\lie_{X_t}\omega_t+\frac{\dd\omega_t}{\dd t}\right).
\end{equation}
In particular,
\begin{equation}\label{eq:derivative-flow-constant-omega}
	\frac{\dd}{\dd t}\psi_t^*\omega=\psi_t^*\lie_{X_t}\omega.
\end{equation}

\section{The real Moser's Path Method}\label{sec:real-moser}

Moser's Path Method \cite{Moser}, or Moser's Trick, is a far-reaching tool used in differential geometry which allows us to find diffeomorphisms between differential forms, provided a vector field satisfying some conditions is found. It can be used to prove Darboux's Theorem \cite{Darboux}, which states that any symplectic form on a $2n$-dimensional manifold is locally symplectomorphic to $\sum_{i=1}^n\dd x_i\wedge\dd y_i,$ where $(x_1,y_1,\ldots,x_n,y_n)$ are the standard coordinates on $\R^{2n}$ centered at the origin. Darboux's result can be generalized from neigborhoods of points to neighborhoods of compact submanifolds, this being the content of the Darboux-Weinstein's Theorem, proved by Alan Weinstein \cite[Theorem 4.1]{Weinstein-submanifolds}. The precise statement of Moser's Path Method is as follows.

Let $\ell$ and $k$ be integers such that $0\le k\le \ell$ and $\ell\ge 1$. Let $M$ be an $\ell$-dimensional smooth manifold, $m\in M$ and $\{\omega_t\}_{t\in[0,1]}$ a family of smooth $k$-forms on $M$. Let $\lie$ denote the Lie derivative. Suppose that there exists a smooth time-dependent vector field $\{X_t\}_{t\in[0,1]}$ on $M$ such that
$\frac{\dd}{\dd t}\omega_t+\lie_{X_t}\omega_t=0$
and $X_t(m)=0$ for every $t\in[0,1]$. Then by \cite[Theorem 2]{Moser} there exists an open neighborhood $U$ of $m$ and a smooth family of diffeomorphisms $\{\psi_t:U\to \psi_t(U)\subset M\}_{t\in[0,1]}$ such that $\psi_t^*\omega_t=\omega_0$ for every $t\in[0,1]$.

The condition that $X_t(m)=0$ for every $t\in[0,1]$ can be replaced by $M$ being compact, and in this case one obtains a smooth family of diffeomorphisms $\{\psi_t:M\to M\}_{t\in[0,1]}$. Compactness is delicate to use in the $p$-adic category and for the purpose of the present paper we work with the first condition ($X_t(m)=0$ for every $t\in[0,1]$).

Indeed, Moser's idea was the following. If either $M$ is compact or $X_t(m)=0$ for all $t$, this vector field defines a flow on the entire interval $[0,1]$. Let $\psi_t$ be this flow. Then we have that
\[\frac{\dd}{\dd t}\psi_t^*\omega_t=\psi_t^*\lie_{X_t}\omega_t+\psi_t^*\frac{\dd}{\dd t}\omega_t=\psi_t^*\left(\frac{\dd}{\dd t}\omega_t+\lie_{X_t}\omega_t\right)=0.\]
This means that $\psi_t^*\omega_t$ is constant, and for $t=0$ is $\omega_0$, hence it is always $\omega_0$. The trick is also applicable in more general situations even when $M$ is not compact, see \cite[Section 2]{CPT}.

In Theorem \ref{thm:moser}, we extend Moser's Path Method to the case where the real field is replaced by the field of the $p$-adic numbers $\Qp$ and the functions involved are analytic. As stated in Section \ref{sec:moser-intro}, we need to ensure some convergence of the solution of an initial value problem, which makes the statement of Moser's Path Method more technical and its proof more complicated. If we used smooth functions instead of analytic ones in the $p$-adic case, Moser's Path Method would not work at all, because a crucial part of it is the fact that, if the derivative of a smooth function is zero, then the function is constant, and there are $p$-adic smooth functions from $\Qp$ to $\Qp$, such as
$f(x)=\sum_{n=\ord_p(x)}^\infty a_np^{2n}$
where
$x=\sum_{n=\ord_p(x)}^\infty a_np^n$
is the $p$-adic expansion of $x$, with zero derivative everywhere and which are injective.

\section{Serre's classification in $p$-adic analytic geometry}\label{sec:serre}

In what follows, when we say that two $p$-adic manifolds are ``diffeomorphic'' we always mean that they are so by means of a $p$-adic analytic diffeomorphism.

\begin{theorem}[Serre {\cite[Th\'eor\`eme 1]{Serre}}, classification of compact $p$-adic analytic manifolds]\label{thm:serre}
	Let $d$ be a positive integer. Let $p$ be a prime number. Let $\mathcal{N}_{d,k}$ be the $p$-adic analytic manifold given as follows:
	\[\mathcal{N}_{d,k}=\bigcup_{i=1}^k\ball\left(\left(\sum_{j=0}^{c(i)}a_j(i)p^{-1-j},0,\ldots,0\right),1\right)\]
	where
	\[i=\sum_{j=0}^{c(i)}a_j(i)p^j\]
	is the base $p$ expansion of $i$, for $i\in\N$.
	\begin{enumerate}
		\item Let $M$ be a compact $d$-dimensional $p$-adic analytic manifold. Then there exists a positive integer $k$ such that $M$ is diffeomorphic to $\mathcal{N}_{d,k}$.
		\item The $p$-adic analytic manifolds $\mathcal{N}_{d,k}$ and $\mathcal{N}_{d,k'}$ are diffeomorphic if and only if $k\equiv k'\mod p-1$.
	\end{enumerate}
\end{theorem}

In the noncompact case, the classification of Theorem \ref{thm:serre} can be extended to the case where the manifold is second-countable. Indeed, if a $p$-adic analytic manifold is second-countable, then it is paracompact, and by Schneider \cite[Proposition 8.7]{Schneider}, it is diffeomorphic to a union of balls, which must be a countable union. Hence, any second-countable (in particular, any compact) $d$-dimensional $p$-adic analytic manifold is diffeomorphic either to $(\Qp)^d$ or to the disjoint union of at most $p-1$ copies of $(\Zp)^d$.

Theorem \ref{thm:serre} has a consequence in the case where $M$ is endowed with a $p$-adic volume form, that is, a $p$-adic analytic differential $d$-form which never vanishes. The typical cases are subsets of $(\Qp)^d$, where we have the standard volume form $\dd x_1\wedge\dd x_2\wedge\ldots\dd x_d$, where $(x_1,\ldots,x_d)$ are the standard coordinates on $(\Qp)^d$, and, for $d=2n$, all $2n$-dimensional $p$-adic analytic symplectic manifolds $(M,\omega)$, where $\omega^n$ is a volume form. In this case, we can define a $p$-adic volume as the $p$-adic measure with respect to this volume form, see for example \cite[p. 14-15]{Weil} and \cite{Popa}.

\begin{corollary}\label{cor:serre}
	Let $M$ and $M'$ be compact $d$-dimensional $p$-adic analytic manifolds endowed with a volume form and
	with the same $p$-adic volume. Then $M$ is diffeomorphic to $M'$.
\end{corollary}

\section{Comparison with Milnor-Husemoller book}\label{sec:comparison}

Corollary 3.5 in Chapter I of \cite{MilHus} states that \textit{if $R$ is either a Dedekind domain or a local ring, then every symplectic inner product space over $R$ is free, and possesses a symplectic basis.} In particular this holds when $R$ is a field, and it can be applied to the case when $R=\R$, giving the well known result that any linear symplectic form in $\R^{2n}$ has a symplectic basis. It can also be applied when $R=\Qp$, giving the corresponding result in the $p$-adics \cite[Theorem A.1]{CrePel-williamson1}, which actually was the first step of the proof of Theorem \ref{thm:darboux2}.

However, this cannot be applied to deduce Darboux's Theorem, neither in the real nor in the $p$-adic setting. If we wanted to deduce it, we would need to restate the concept of a symplectic form on a (real or $p$-adic) manifold in terms of a symplectic inner product space. We can do it by considering the $2$-form $\omega$ as a bilinear function $\X(M)\times\X(M)\to\Omega^0(M)$, which is a symplectic inner product. At this point we would have such a product over $\Omega^0(M)$, when we need either a Dedekind domain or a local ring, and $\Omega^0(M)$ is neither: it is not a domain because the product of two nonzero smooth functions can be zero, and it is not a local ring because it has infinitely many maximal ideals (for any $m\in M$, the ideal of the functions which vanish at $m$ is maximal).

This statement of symplectic forms on manifolds in terms of symplectic inner product spaces actually appears in \cite[Section V.2]{MilHus}, but only in the real case. Actually, it is deduced that such a symplectic basis of $\X(M)$ cannot exist if $M=\mathrm{S}^2$, which means that something else is needed about $M$ (namely, restricting to a small neighborhood of $m$) in order to have a symplectic basis.

\end{document}